\documentclass[11pt]{article}   
\usepackage{amssymb,amscd,latexsym}   
\usepackage{amsmath}
\usepackage{amsthm}
\usepackage{mathdots}
\usepackage[pagebackref]{hyperref}
\usepackage{color}
\usepackage[all]{xy}
\newcommand{\rar}{\rightarrow}
\newcommand{\lar}{\longrightarrow}
\newcommand{\llar}{-\kern-5pt-\kern-5pt\longrightarrow}
\newcommand{\surjects}{\twoheadrightarrow}

\newtheorem{Theorem}{Theorem}[section]
\newtheorem{Lemma}[Theorem]{Lemma}
\newtheorem{Corollary}[Theorem]{Corollary}
\newtheorem{Proposition}[Theorem]{Proposition}
\newtheorem{Remark}[Theorem]{Remark}
\newtheorem{Example}[Theorem]{Example}

\newtheorem{Definition}[Theorem]{Definition}
\newtheorem{Question}[Theorem]{Question}

\def\sqr#1#2{{\vcenter{\hrule height.#2pt
        \hbox{\vrule width.#2pt height#1pt \kern#1pt
            \vrule width.#2pt}
        \hrule height.#2pt}}}
\def\phi{\varphi}

\DeclareMathOperator{\Image}{Im}
\DeclareMathOperator{\coker}{Coker}

\DeclareMathOperator{\depth}{depth}
\DeclareMathOperator{\rank}{rank}
\DeclareMathOperator{\Ht}{ht}
\DeclareMathOperator{\reg}{reg}

\DeclareMathOperator{\grade}{grade}

\def\xx{{\bf x}}
\def\yy{{\bf y}}

\def\TT{{\bf T}}

\def\uu{{\bf u}}

\def\fm{{\mathfrak m}}

\def\Ht{{\rm ht}\,}
\def\depth{{\rm depth}\,}

\def\ker{{\rm ker}\,}
\def\grade{{\rm grade}\,}

\def\spec#1{{\rm Spec}\, (#1)}

\def\restr{{\kern-1pt\restriction\kern-1pt}}

\def\pp{{\mathbb P}}

\begin{document}
\begin{center}
	{\Large{\bf\sc   Cohen--Macaulay ideals of codimension two and the geometry of plane points}}
	\footnotetext{AMS Mathematics
		Subject Classification (2010   Revision). Primary 13A02, 13A30, 13D02, 13H10, 13H15; Secondary  14E05, 14M07, 14M10,  14M12.} 
	
\vspace{0.1true in}

{To the memory of Wolmer Vasconcelos and Tony Geramita} \\

	\vspace{0.3in}
{\large\sc Dayane Lira}\footnote{
	Partially under a post-doc fellowship from INCTMAT/Brazil (160944/2022-8)} \quad
{\large\sc Geisa Oliveira}\footnote{Under a PhD fellowship from CAPES/Brazil (88887.627699/2021-00)} \quad	
	{\large\sc Zaqueu Ramos}\footnote{Partially
		supported by a CNPq grant (304122/2022-0)} \quad
	{\large\sc Aron  Simis}\footnote{Partially
		supported by a CNPq grant (301131/2019-8).}

\end{center}

\begin{abstract}
We consider  classes of codimension two Cohen--Macaulay ideals over a standard graded polynomial ring over a field. 
We revisit Vasconcelos' problem on $3\times 2$ matrices with homogeneous entries and describe the homological details of Geramita's work on plane points. 
An additional topic is the homological discussion of minors fixing a submatrix in the context of a perfect codimension two ideal.
A combinatorial outcome of the  results is a proof of the conjecture on the Jacobian ideal of a hyperplane arrangement stated by Burity, Simis and Toh\v{a}neanu.
The basic drive behind the present landscapes is a thorough analysis of the related Hilbert--Burch matrix,  often without assuming equigeneration, linear presentation or even the popular $G_d$ condition of Artin--Nagata.  
\end{abstract}


\section*{Introduction}

The main object of this work is a perfect ideal $I$ of codimension two in a polynomial ring $R$  over a field, not always assuming at the outset that $I$ is a generically complete intersection, in the hope of getting some additional significant results.

Since any such ideal admits a short free resolution, it can
 essentially be studied through the behavior of its syzygies as columns of a {\em Hilbert-Burch matrix} containing quite a bit of algebraic/homological invariants.
This line of work has been followed by several authors, starting with the landmark papers by Morey and Ulrich (\cite{Morey}, \cite{MorUl1996}, \cite{Ulrich}), followed by the subsequent intervention of T\`ai (\cite{Tai2001}), and a bit later,  Madsen (\cite{Madsen}), Lan (\cite{Lan}) and Doria--Ramos--Simis (\cite{linpres2018}), and, more recently, Costantini, Price and Weaver (\cite{CPW2024}). 

Of a quite different nature, large considerations have way back been given to the search of canonical Hilbert--Burch matrices with entries in two variables that parametrize monomial (primary) ideals. This endeavor has to do with describing the Hilbert scheme of points and has been pursued by several experts in various directions, such as Bialynicki-Birula (\cite{Bi-Bi1}, \cite{Bi-Bi2}), Brian\c con (\cite{Bria}), Iarrobino (\cite{Iarrob}), and, more recently, Conca and Valla (\cite{ConcVal}), Constantinescu (\cite{Const}), and Homs and Winz (\cite{Homs-Win}).

Additional work has carried on other aspects of the nature of a Hilbert--Burch matrix, making it too many to report on here.
From what we could gather, the approach over here is somewhat different from the above collection, where often the ideal is equigenerated or linearly presented.  We do not deal with Gr\"obner bases, hence monomial ideals are not particularly distinguished, although they surely play a role behind the curtains.
We have picked up three topics of slightly different colors in which similar techniques seem to apply in a variety of landscapes.

One of these issues leads to obtaining structure theorems concerning the generation and homology of the ideal $I$ and the subideal generated in its initial degree, often read off as a certain ideal of maximal minors fixing a convenient submatrix of the Hilbert-Burch matrix. Some effort has been employed to finding when the corresponding linear system on the initial degree defines a birational map onto the image.

Supplementary disclosure of previous work on the subject is eventually indicated in the details of the various subsections.
Thus, we move on to the contents of the sections.

Section 2 stages as a preamble, concerning the case of $3$-generated such ideals $I$ in $k[x,y,z]$ -- geometrically, one is looking at almost complete intersections of not necessarily reduced points in $\pp^2_k$, if it is of any help.
Although the corresponding Hilbert-Burch matrix $\phi$ looks pretty simple in shape, not so much certain relevant properties of the ideal generated by its entries. In fact, we will be interested in the case where this ideal is weirdest, namely, the case where it has height two, thus implying that $I$ is not an ideal of linear type.
This motivates looking at the defining ideal $\mathcal{I}$ of the blowup of $V(I)\subset \pp^2_k$, i.e., of the Rees algebra $\mathcal{R}_R(I)$ of $I$.
In the case where $\phi$ has a linear syzygy, and under an additional condition on the entries of the ``second'' syzygy, the ideal $\mathcal{I}$ defines a de Jonqui\`eres structure, hence is essentially generated by a certain sequence of {\em downgraded} biforms as a consequence of parts of \cite{De_Jonq}.
This will in particular entail that $\mathcal{R}_R(I)$ is Cohen--Macaulay if and only if the second syzygy has degree $\leq 2$.

If all generating syzygies are of degree $\geq 2$ then  we face a  different landscape, so we resort to the so-called Sylvester forms, of which quite a bit is spread-out through the literature. 
One basic result here (Proposition~\ref{Zaq}) gives conditions under which the Rees algebra $\mathcal{R}_R(I)$ of $I$ is itself perfect of codimension two, hence Cohen--Macaulay.
Here, the defining ideal of $\mathcal{R}_R(I)$ is generated by three (bigraded) generators, two of which are the obvious ones coming from the syzygies of $I$ and the additional one a certain Sylvester form.
We discuss examples showing that some of the assumptions of this proposition are essential to the result, by otherwise exhibiting additional Sylvester forms as minimal bigraded generators.
We are presently not aware of an explicit condition under which $\mathcal{R}_R(I)$ is Cohen--Macaulay. It is quite feasible that reading the bidegrees of a few available defining minimal generators triggers necessary conditions in order to build up a Hilbert--Burch matrix whose minors generate the defining ideal of $\mathcal{R}_R(I)$ (cf. some of the ingredients in the proof of \cite[Proposition 3.1]{HasSim2017}).
A curious reader may find some interest in trailing this path toward a definite outcome.

The theme of this section has been considered way back by Wolmer Vasconcelos and the present last author.

In the next section we assume that the ground ring is an arbitrary polynomial ring $R=k[x_1,\ldots,x_d]$ over an infinite field $k$ and allow the codimension two perfect ideal $I\subset R$ to be generated in two degrees.
In this case the corresponding Hilbert-Burch matrix has the shape 
 \begin{equation}\label{basic_matrix}
	{\phi}=\left[\begin{matrix}
		\,		{\Phi}_1\,\\
		\hline \,
		\,	\Phi_2 \,
	\end{matrix}\right],
\end{equation}
where $\Phi_1$ is an $a\times (n-1)$ submatrix  whose row entries have degree $\epsilon_1$ and $\Phi_2$ is an $(n-a)\times (n-1)$  matrix whose row entries  have degree $\epsilon_2$, for choices of integers $1\leq a\leq n-1$ and $1\leq \epsilon_2\leq \epsilon_1$.
Among our main players, is the ideal $J\subset I$  generated by the maximal minors $f_1,\ldots,f_a$ of $\phi$  fixing the rows of the submatrix $\Phi_2$. Note that $J$ is generated in the initial degree $D-\epsilon_1$ of $I$, where $D=a\epsilon_1+(n-a)\epsilon_2$.

A free resolution of $R/J$ was obtained formerly in \cite[Theorem D]{AnSi1986}, but here we go deeper into the graded aspects of the setup. For this, we work on an encore of the Buchsbaum-Rim complex as it fits our purpose to pursue finer properties of the pair $J\subset I$.
We digress on the minimal graded free resolutions of both $R/J$ and $I/J$, with an eye for the case where the ideal of minors $I_{n-a}(\Phi_2)$ has height $a$, which forces $a\leq d$.
The special cases where $a=3$ and $a=d$ (maximum) are treated in detail. In the first case, we show that the symmetric algebra of $J$ is Cohen--Macaulay and give its minimal graded free resolution.
In the second case, assuming moreover that $\Phi_2$ is nonzero linear (hence, $n\geq d+1$), we show our main theorem of the section, composed of two parts: first, if $\Phi_1$ is also linear then $J$ is a reduction of $I$; second,  $J$ is of linear type if and only if $I$ satisfies property $G_d$ of Artin--Nagata (\cite{AN}).

As a consequence of the second of these results we are able to prove the standing \cite[Conjecture 3.1]{BuSiTo2022}  in the case of a generic arrangement.

In the last section we focus on the codimension two perfect ideal defining a finite set of plane reduced points. 
This  is a longtime cherished topic that has been tackled from various aspects, both algebraic and geometric.
We chose to follow closely the early scripts of Tony Geramita  (to whom this paper is dedicated) and co-authors, tough eventually our approach stresses the homological side of the questions. 

Let $\mathfrak{G}_n\subset (\pp^2)^n$ denote the set of the ($n$-tuples of) points in generic $n$-position, and let 
 $\mathfrak{TG}_n\subset\mathfrak{G}_n\subset (\pp^2)^n$ denote the subset of the  points whose corresponding reduced ideals $I\subset R$ satisfy the equality $\dim_k(R_1 I_s)=\min\{3\dim_k I_s,\dim_k I_{s+1}\}$, with $s={\rm indeg}(I)$.
 It was proved  that $\mathfrak{TG}_n$ contains a non-empty open subset  (\cite[Theorem 2.6]{GM}, a result that implies the conjecture of L.  Roberts stated in \cite{LRoberts} -- the latter having been extended to arbitrary dimensions in \cite{TrVa1989}).
 A point (tuple) belonging to the subset  $\mathfrak{TG}_n$ will be said to be in {\it tight generic $n$-position}.	

One knows that $s:={\rm indeg}(I)$ is the least integer such that $n< {s+2\choose2}.$ In particular, ${s+1\choose2}\leq n$, so one can write $n={s+1\choose2}+h$ for some $0\leq h\leq s$.
We give explicit minimal graded free resolutions of $I$ by distinguishing between the cases as to whether one is in sector $0\leq h\leq s/2$ or $s/2< h \leq s$.
We then proceed to examining certain facets of the two sectors to enhance their very differences.

In the first sector, assuming that the  inequality $h\leq s/2$ is strict and $s\geq 3$, we show that the  rational map $\pp^2 \dasharrow \pp^{s-h}$ defined by the linear system $I_s$ is birational onto the image.
As a consequence -- always assuming that $h\leq s/2-1$, and moreover the minimal number of generators of $I$ is at least $4$ -- we derive that the Rees algebra of $I$ is Cohen--Macaulay if and only if $h=0$ (i.e., $I$ is linearly presented). Moreover, this equivalence is also established in terms of the Cohen--Macauleyness of the special fiber $\mathcal{F}(I)$ and its minimal graded free resolution.

The behavior in the sector $s/2< h \leq s$ requires a different approach, by and large connected closely to the one in Section~\ref{Section3}.
We give a new proof of a result of Geramita regarding the case where the points are in uniform $n$-position (a stronger assumption that generic $n$-position), namely, any element of the linear system of $I$ in the initial degree is necessarily irreducible.
This has a certain impact on the subsequent findings.
Here the main theorem of this sector (Theorem~\ref{main-sector2})  concerns the homological behavior of the ideal $J:=(I_s)$ under the uniform $n$-condition. Among the items of this theorem we find that $J$ is an ideal of linear type and that its Rees algebra is Cohen--Macaulay, while the rational map defined  by a minimal set of generators of $J$ is {\em not} birational onto its image -- the argument for the latter rests on applying the shapes of free resolutions discussed in the previous section.
It should be said that \cite{Tai2001} contains a thorough study of the Rees algebra of the ideal $(I_t)$ for sufficiently large $t$.

For the quick reader's facility, the main results are Theorem~\ref{main-thm}, Theorem~\ref{Arrang_conj},  
Theorem~\ref{res-ideal-gen-points}, 
Theorem~\ref{strict_inequality}, and Theorem~\ref{main-sector2}.

\section{Notation}\label{prelims}

Let $R$ be a Noetherian ring and let $I = (f_1 , \ldots , f_n) \subset R$ denote an ideal of $R$.

We consider in this work the following graded algebras associated with the pair $(R,I)$:

\begin{enumerate}
	\item[$\bullet$] The symmetric algebra  $\mathcal{S}_R(I)=\bigoplus_{i\geq 0}\mathcal{S}_{R,i}(I)$, as the direct sum of the symmetric powers of $I.$
	\item[$\bullet$] The Rees algebra $\mathcal{R}_R(I)=\bigoplus_{i\geq 0}I^i\simeq R[f_1t,\ldots,f_nt]\subset R[t]$, as the direct sum of the powers of $I.$ 
	
	If, moreover, $R$ is local with maximal ideal $\fm$ (or standard graded with maximal irrelevant ideal $\fm$), we set:
	\item[$\bullet$] The special fiber  (or fiber cone) $\mathcal{F}_{R}(I):= \mathcal{R}_R(I)/\mathfrak{m} \mathcal{R}_R(I).$
\end{enumerate}

For a graded algebra $A$ we denote by $e(A)$ its (Hilbert) multiplicity.

The algebras $\mathcal{S}_R(I)$ and $\mathcal{R}_R(I)$ are related by a natural $R$-algebra surjection of graded  $R$-algebras $$\alpha:\mathcal{S}_R(I)\surjects\mathcal{R}_R(I).$$
The kernel $\mathcal{A}$ of this map is the $R$-torsion submodule of $\mathcal{S}_R(I)$ provided $I$ has grade $\geq 1$. 
Let $\mathcal{L}$ and $\mathcal{J}$ respectively denote  the kernel of the following  surjectives $R$-algebra homomorphism  
$$R[y_1,\ldots,y_n]\surjects \mathcal{S}_R(I), \,y_i\to f_i, \quad\mbox{and}\quad
R[y_1,\ldots,y_n]\surjects \mathcal{R}_R(I), \,y_i\to f_it.$$
They are called {\em presentation ideals} of the respective two algebras.
Obviously, the second of the above  homomorphisms factors through the two homomorphism $$R[y_1,\ldots,y_n]\to\mathcal{S}_R(I)\stackrel{\alpha}\surjects\mathcal{R}_R(I).$$ In particular, $\mathcal{L}\subset \mathcal{J}$ and $\mathcal{A}\simeq \mathcal{J}/\mathcal{L}.$ It is well known that $\mathcal{L}$ can be obtained from a finite presentation of $I$ as in \eqref{syzI}. In fact, $\mathcal{L}=I_1(\yy\cdot \phi),$ where $\yy$ is the matrix $[y_1 \cdots y_n]$ and the notation $I_t(A)$  means the ideal generated by all $t$-minors of the matrix $A.$

A presentation ideal of the special fiber over $k$ is easily established from a presentation ideal of the Rees algebra.
In the special case where $I = (f_1 , \ldots , f_n) \subset R$ is an equigenerated ideal in a standard graded ring over a
field $k$ then, the kernel $Q$ of the surjection of graded $k$-algebras
$$k[y_1, \ldots,y_n] \to k[f_1t, \ldots, f_nt],\, y_i \to f_it $$
is a presentation ideal of $\mathcal{F}_{R}(I)$.

Consider in addition a free $R$-module presentation of $I$:
\begin{equation}\label{syzI}
R^{m}\stackrel{\phi}\lar R^n\to I\to 0.
\end{equation} 

The following notion was introduced in \cite{AN}:

\begin{Definition}\rm
Given an integer $s\geq 1$, the ideal $I$ (or the matrix $\phi$) satisfies condition $G_s$ if, equivalently:

(i) $\mu(I_p)\leq \Ht p$ for every prime $p\in V(I)$ with $\Ht p\leq s-1.$

(ii) $\Ht I_{j}(\phi)\geq n-j+1$ for all $n-s+1\leq j\leq n-1.$
\end{Definition}

As is known, this sort of condition and its analogues have an impact on certain dimension theoretic aspects.
Throughout this part, $R$ will denote a Noetherian ring of finite Krull dimension $d$.
In this context, the most notable cases of the above condition have $s=d$ and $s=d+1$, the second also denoted $G_{\infty}$.
Here we intend to consider lower cases of $s$, to see how they impact in a few situations.

The symmetric algebra above is very sensitive to it (see, e.g., \cite[Section 7.2.3 and ff]{SimisBook}).

When no confusion arises, we will omit the subscript $R$ in the notation of the symmetric algebra, Rees algebra, and special fiber.


 
Our main player is an $n\times (n-1)$  matrix $\phi$ with homogeneous entries in a standard graded polynomial ring $R$ over an infinite field $k$.
 Though not very precisely, one refers to these matrices as {\em Hilbert--Burch matrices}, a terminology currently used at large.
 Upon the usual condition that the ideal $I_{n-1}(\phi)$ has height $2$ (maximal possible), we are dealing with the syzygy matrix of  a perfect ideal of height $2$.

 \section{A primer of $3\times 2$ matrices over $k[x,y,z]$}

  Basic research on $3\times 2$ matrices over a polynomial ring $k[x,y]$ has been carried out by many authors (see \cite{BuJou2003}, \cite{Cox2008}, \cite{CHW2008}, \cite{Syl1}, to mention a few).

 We henceforth take up the ternary case.
 Thus,  assume that  $R=k[x,y,z]$, endowed with the standard grading in which $\deg x=\deg y=\deg z=1$.
 The case where the entries of the matrix are linear forms has been considered, even for $n\times (n-1)$ matrices, with arbitrary $n$ (see, e.g., \cite{SymbPowBir2014}, \cite{Lan}, \cite{linpres2018}, \cite{CPW2024}).
 For entries of higher degrees, one or another among the standing hypotheses of \cite[Theorem 1.2]{MorUl1996} are not available.
 Thus, taking for granted  that most is known in this low dimension may be delusional. 
 Even in this narrow environment, one topic that seems to require further non-trivial work is the structure of the associated algebras, such as the  Rees algebras and the special fiber.

 To move on, we fix the notation. Let $\phi=(a_{i,j})_{1\leq i\leq 3\atop {1\leq j\leq 2}}$ denote a $3\times 2$ matrix over $R$, where $1\leq \deg a_{i,1}:=d_1\leq d_2:=\deg a_{i,2}$, for all $i$.
 Thus, the ideal $I:=I_2(\phi)\subset R$ of $2$-minors is a homogeneous ideal equigenerated in degree $d:=d_1+d_2$.
 We assume in addition that this ideal has codimension at least two (hence, exactly two). Therefore, $I$ is a perfect ideal, i.e., $R/I$ is Cohen--Macaulay.
 
 
  
 In addition, set $\fm=(x,y,z)$ and $\xx=\{x,y,z\}$.
 To fix ideas, let $\Delta_{i}$ denote the $2$-minor of $\phi$ excluding the $i$th row.
 Mapping a polynomial ring $S=R[t_1, t_2, t_3]$ onto the $R$-algebra $R[It]=R[\Delta_{1}t,\Delta_{2}t,\Delta_{3}t]\subset R[t]$ of $I$ via $t_i\mapsto \Delta_i t$, let $\mathcal{I}\subset S$ denote the corresponding kernel -- referred to as a {\em homogeneous defining ideal} of  $R[It]$.
 Clearly, $R[It]$ is isomorphic to the Rees algebra $\mathcal{R}_R(I)$ as introduced in Section~\ref{prelims}.
 
 Notable is the fact that, since $I$ is equigenerated, then (pretty generally) $R[It]$ admits a natural bigrading induced by the standard bigrading of the polynomial ring $S=k[x,y,z,t_1,t_2,t_3]$, and moreover, it admits the special fiber $\mathcal{F}(I)$ as a  direct summand.
 Finally, let $C(\mathcal{I})$ stand for the homogeneous ideal of $R$ generated by the $\xx$-coefficients of every element of $\mathcal{I}$.

The following questions were first raised by W. Vasconcelos:
 \begin{Question}\label{q2}
 	\rm
 	
 	(a) When is $C(\mathcal{I})=\fm$?
 	
 	(b) When is $C(\mathcal{I})$ an $\fm$--primary ideal?
 	
 	(c) When is the analytic spread $\ell(I)$ of $I$ maximum ($=3$)?
 	
 	(d) When is  the rational map
 	${\mathbb P}^2 \dasharrow {\mathbb P}^2$ defined by generators of $I$ 
 	birational?
 \end{Question}
 
 Trivially, affirmative (a) $\Rightarrow$ affirmative (b). Also, since $I$ is equigenerated, then $\ell(I)$ is the Krull  dimension of the $k$-algebra $k[I]\simeq k[It]$, hence $\ell(I)\leq 2$ implies a nonzero strict polynomial relation, i.e., $C(\mathcal{I})=(1)$.
 Therefore, affirmative (b) $\Rightarrow$ affirmative  (c). Moreover, clearly affirmative (d) $\Rightarrow$ affirmative (c).
 Finally, affirmative (a) is a good deal in the direction of proving that $R[It]$ is regular locally in
 codimension one. Recall that the latter is equivalent to birationality when $\dim R/I=0$, but certainly not in the present context with $\dim R/I=1$.
 
 In this part we will be essentially interested in questions (c) and (d).

 \subsection{The ideal generated by the entries} 
 
 Since obviously $I\subset I_1(\varphi)$, then $I_1(\varphi)$ has codimension $\geq 2$. 
 If  $I_1(\varphi)$ has codimension $3$ then $I$ is an ideal of linear type (\cite[Lemma 3.1]{NejadSimis2011}) -- and actually must have a syzygy whose coordinates form a regular sequence (\cite[Theorem 2.1]{Toh2013}). Thus, the structure of the ring $R/I$ and its main associated algebras are well-know.
 
 Therefore,  assume henceforth that $I_1(\varphi)$ has codimension two. In particular, every minimal prime of $R/I_1(\varphi)$ is a minimal prime of $R/I$ but in general not the other way around. Actually, typically $(x,y,z)$ may be an embedded prime of $I_1(\varphi)$, that is, $\depth R/I_1(\varphi)=0$.
 We emphasize for the record that, under the latter assumption, $\phi$ does not satisfy the property $G_3$.
 
 \subsubsection{A column of degree one}

 Here,  the entries of the first column of $\phi$ are linear forms.
 By our standing assumption on the codimension of  $I_1(\varphi)$, the three forms are $k$-linearly dependent.
 By an elementary row operation (over $k$) -- which implies a $k$-linear change of the generators of $I$ -- we may assume that the first column has the shape $(\ell_1,\ell_2, 0)^t$, where $\ell_1,\ell_2$ are $k$-independent linear forms.
 
 Since $(\ell_1,\ell_2)$ is a prime ideal, we must have  $I_1(\varphi)=(\ell_1,\ell_2)$.
 Further, by a change of variables, we may as well assume that $\ell_1=x, \ell_2=y$.

 Thus, we may assume that the matrix is of the form
 \begin{equation}\label{linear_syzyzy}
 	\left[\begin{matrix} x & \gamma_1 \\
 		y & \gamma_2 \\
 		0 & \gamma_3
 	\end{matrix}\right] ,
 \end{equation}
 where $\gamma_1, \gamma_2,\gamma_3\in (x,y)$  since $I_1(\phi)=(x,y)$.
 Thus, $I=(x\gamma_3,y\gamma_3, x\gamma_2-y\gamma_1).$
 Since we are assuming that $I$ has codimension two, then $\gcd (\gamma_3, x\gamma_2-y\gamma_1)=1$.
 Moreover, we will assume that one at least among $\gamma_1, \gamma_2, \gamma_3$ involves effectively $z$.
 
 Note that $I$ is generically a complete intersection except at $(x,y)$ and $I:(x,y)^{\rm sat}=J:(x,y)^{\rm sat}$, where $J:=(\gamma_3, x\gamma_2-y\gamma_1)\supset I$ admits same associated primes as $I$.
 
 For the main result on this matrix format, we recall the terminology of \cite{De_Jonq}.
 
 Quite generally, let $\mathfrak{J}:\pp^n\dasharrow \pp^n$  denote a rational map. 
For maximum clarity, we will distinguish $\pp^n_x={\rm Proj}(k[x_1,\ldots,x_{n+1}])$ (source) and $\pp^n_y={\rm Proj}(k[y_1,\ldots,y_{n+1}])$ (target).
 
 Given a rational map in one dimension less $\mathfrak{F}:\pp^{n-1}\dasharrow \pp^{n-1}$, where,  similarly, $\pp^{n-1}_{x'}={\rm Proj}(k[x_1,\ldots,x_{n}])$ (source) and $\pp^{n-1}_{y'}={\rm Proj}(k[y_1,\ldots,y_{n}])$ (target), 
 let $\pi_x: \pp^n_x\surjects \pp^{n-1}_{x'}$  and $\pi_y: \pp^n_y\surjects \pp^{n-1}_{y'}$ denote the natural coordinate projections, themselves rational maps.
 
 \begin{Definition}\rm
 	We say that $\mathfrak{J}$ and $\mathfrak{F}$ are {\em confluent} if
 	\begin{equation}\label{basic_assumption_deJonq}
 		\mathfrak{F}\circ \pi_x=\pi_y\circ \mathfrak{J}
 	\end{equation}
 	as rational maps from $\pp^n_x$ to $\pp^{n-1}_y$.
 \end{Definition}
A similar notion goes by reversing variables.
 
\begin{Lemma}\label{Jonq_basic_form} {\rm (\cite[Lemma 1.2]{De_Jonq})}
	Let $\{g_1,\ldots, g_n \}\subset k[x_1,\ldots,x_{n}]$ be forms of the same degree with no proper common factor and let
	$\mathfrak{J}$ and $\mathfrak{F}$ be confluent rational maps as above.  Then $\mathfrak{F}=(g_1:\cdots : g_n)$  if and only if $\mathfrak{J}=(fg_1 :\cdots : fg_n:g)$ for suitable forms $f,g\in k[x_1,\ldots,x_{n+1}]$ satisfying $\gcd(f,g)=1$.
\end{Lemma} 

\begin{Proposition}\label{generalized_deJonq_setup} {\rm (\cite[Proposition 1.6]{De_Jonq})}
	Suppose that $\mathfrak{F}$ and $\mathfrak{J}$ as above are confluent, where $\mathfrak{F}=(g_1  :\cdots : g_n)$, with $\gcd(g_1,\ldots,g_n)=1$, and $\mathfrak{J}=(fg_1 :\cdots :fg_n :g)$, with $\gcd(f,g)=1$. 
	The following are equivalent$:$
	\begin{enumerate}
		\item[{\rm (i)}]  $\mathfrak{J}$ is birational.
		\item[{\rm (ii)}]  $\mathfrak{F}$ is birational  and, moreover,  $f,g$ are $x_{n+1}$-monoids at least one of which involves $x_{n+1}$ effectively.
	\end{enumerate}
\end{Proposition}

 We then have:
 \begin{Proposition}\label{deJonq}
 	With the above notation and assumptions$:$
 	\begin{enumerate}
 		\item[{\rm (i)}]	 The rational map $\mathfrak{R}: \pp^2\dasharrow \pp^2$ defined by the generators of $I$ is confluent with the identity map of $\pp^1$.
 		\item [{\rm (ii)}]	 
 		$\mathfrak{R}$ is birational if and only if $\gamma_1, \gamma_2, \gamma_3$ are $z$-monoids, in which case the map is a de Jonqui\`eres map.
 		\item [{\rm (iii)}] Let $I$ be generated in degree $d$.
 		If the equivalent conditions of {\rm (ii)} hold, a presentation ideal of the Rees algebra $\mathcal{R}_R(I)$ of $I$ is minimally generated by $d$ polynomials of bidegrees 
 		$$(1, 1), (d -1, 1), (d -2, 2), \ldots , (1,d -1).$$
 		In particular, $\mathcal{R}_R(I)$ is Cohen--Macaulay if and only if $d\leq 3$.
 	\end{enumerate}
 \end{Proposition}
 \begin{proof} (i) Follows from Lemma~\ref{Jonq_basic_form}.
 
 (ii) Follows from Proposition~\ref{generalized_deJonq_setup} and item (i).
 
 (iii) This is \cite[Theorem 2.7, (i) and (iii)]{HasSim2012}.
 \end{proof}
 
 \begin{Remark}\label{Rees_algebra_downgraded}\rm
 	The present discussion with the $z$-monoid condition is a special case of de Jonqui\`eres maps confluent with the identity.
 	As an alternative to  (iii) above, one has a full description of the presentation ideal of the Rees algebra of the ideal $I$ in terms of certain downgraded sequence of biforms (\cite[Theorem 2.6]{De_Jonq}).
 	In the case of cubics -- i.e., when $\deg \gamma_1= \deg \gamma_2=\deg \gamma_3=2$ -- the defining ideal of the Rees algebra had been previously noted in \cite{Vasc-Hong}.
 \end{Remark}
 
 \subsubsection{The $(x,y)$-primary component of the ideal of entries}
 
 The previous part focused on the case where, somewhat trivially, the ideal $I_1(\phi)$ is a simple linear complete intersection.
  We now focus, more generally, on the case where $(x,y)$ is a minimal prime of $I_1(\phi)$, and the
 $(x,y)$-primary component of $I_1(\phi)$ is the complete intersection $(x^m,y^n)$, for suitable exponents $m\geq 1,n\geq 1$.
  If $k$ is algebraically closed, the minimal primes of $I_1(\phi)$ are always generated by two linear forms, so by a coordinate change, one may assume that one of these primes is $(x,y)$. Thus, our setup is quite natural.
  

 \begin{Proposition}\label{Zaq}
 	Given an integer $m\geq 1$, consider a matrix of the form
 	$$\phi=\left[\begin{matrix}x^m&p_1\\
 		y^m&p_2\\
 		0&p_3\end{matrix}\right],$$
 		satisfying the following hypotheses$:$
 	\begin{enumerate}
 		\item[{\rm (a)}] $p_1,p_2,p_3$ are forms in $R:=k[x,y,z]$ of degree $n+\epsilon$, for some $1\leq n\leq m$ and $\epsilon\geq \max\{m-n,1\}$.
 		\item[{\rm (b)}] The ideal $I=I_2(\phi)$ has height $2$.
 			\item[\rm(c)] The $(x,y)$-primary component of $I_1(\phi)$ is $(x^m,y^n).$
 		\item[{\rm (d)}] If $m=n$, either $\Delta$ or $p_3$ has $z$-degree $\epsilon$.
 	\end{enumerate}
 	Then, letting  $\mathcal{R}(I)$ denote the  Rees algebra of $I$ over $R$, one has$:$
 	    \begin{enumerate}
 	    	\item[{\rm (i)}] $\mathcal{R}(I)$ is Cohen-Macaulay.
 	    	\item[{\rm (ii)}] The bihomogeneous defining ideal of $\mathcal{R}(I)$ is generated by the defining relations of the symmetric algebra of $I$ and a determinantal form of bidegree $(\epsilon,2).$
 	    \end{enumerate} 
 \end{Proposition}
 
 \begin{proof} By (c), in particular, $p_i\in (x^m,y^n)$, $1\leq i \leq3$, so  write
 	
 	$$p_i= \underbrace{\left(\sum_{j_1+j_2+j_3=n+\epsilon-m} a^{(i)}_{j_1,j_2,j_3}x^{j_1}y^{j_2}z^{j_3}\right)}_{=:q_i}x^m+\underbrace{\left(\sum_{j_1+j_2+j_3=\epsilon} b^{(i)}_{j_1,j_2,j_3}x^{j_1}y^{j_2}z^{j_3}\right)}_{=:q'_i}y^n, (1\leq i \leq3),$$
 	for certain $a_{j_1,j_2,j_3},b_{j_1,j_2,j_3}\in k.$  
 	Let  $\mathcal{R}_R(I)\simeq R[t_1,t_2,t_3]/\mathcal{I}$ denote a bigraded presentation of the Rees algebra and let $\{f,g\}\subset \mathcal{I}$ stand for the defining equations of $\mathcal{I}$ induced by the columns of the matrix $\phi$. 
 	Consider the $R$-content matrix of $\{f,g\}$ with respect to the sequence $\{x^m,y^n\}$, that is, the larger matrix below
 	$$\left[\begin{matrix}f&g\end{matrix}\right]=
 {	\renewcommand\arraystretch{1.5}
 	\left[\begin{matrix}t_1&t_2y^{m-n}\\	
 		\displaystyle\sum_{1\leq i \leq 3} q_it_i& \displaystyle\sum_{1\leq i \leq 3} q'_it_i\end{matrix}\right]
 	}
 	\left[\begin{matrix} x^m\\
 		y^n \end{matrix}\right],
 	$$
 	and let $h\in R[t_1,t_2,t_3]$ denote its determinant, a byform of bidegree $(\epsilon, 2)$.
 	
 	Then, $\mathcal{J}:=(f,g,h)\subset \mathcal{I}.$ 
 	Note that $\mathcal{J}$ is an ideal of codimension 2 generated by the 2-minors of the following $3\times 2$ matrix:
 	\begin{equation}\label{H-B2}
 		\renewcommand\arraystretch{1.5}
 		\left[\begin{matrix}t_1&t_2y^{m-n}\\ 
 			\displaystyle\sum_{1\leq i \leq 3} q_it_i&\displaystyle\sum_{1\leq i \leq 3} q'_it_i\\
 			-y^n&x^m\end{matrix}\right].
 	\end{equation}
 	In particular, $R[t_1,t_2,t_3]/\mathcal{J}$ is a Cohen-Macaulay ring.
 	The rest of the procedure consists in showing that $\mathcal{J}$ is a prime ideal.
 	
 	{\sc Claim.} $x$ is a regular element over $R[t_1,t_2,t_3]/\mathcal{J}$. 
 	
 	Otherwise, there is an associated prime $P$ of $R[t_1,t_2,t_3]/\mathcal{J}$ such that $x\in P.$ In particular, since $f\in\mathcal{J}\subset P$ then  $t_2y^m\in P.$ Thus, $t_2\in P$ or $y\in P.$
 	We now show that either of these assumptions leads to a contradiction.
 	
 	\smallskip
 	
 	\underline{ Case 1.} $t_2\in P.$ 
 	
 	\smallskip
 	
 	Since $g\in\mathcal{J}\subset P$ then $t_1\delta\in P$, where $$\delta=\left(\sum_{j_2+j_3=\epsilon}b^{(1)}_{0,j_2,j_3}y^{j_2}z^{j_3}\right) t_1 +\left(\sum_{j_2+j_3=\epsilon}b^{(3)}_{0,j_2,j_3}y^{j_2}z^{j_3}\right)t_3,$$ 
 	a nonzero form since, otherwise, $I\subset (x)$, contradicting assumption (a). As $t_1\delta$ does not involve either $x$ or $t_2,$ then $\{x,t_2,t_1\delta\}$ is a regular sequence in $P$, which is nonsense because $P$ itself has codimension $2$.
 	
 	\smallskip
 	
 	\underline{Case 2.} $y\in P.$
 	
 	\smallskip
 	
 	In this case, the inclusion 
 	$h\in\mathcal{J}\subset P$ implies that 
 	{\small
 		$$\theta:=\det
 	\renewcommand\arraystretch{1.5}
 	 \left[\begin{matrix} t_1&t_2y^{m-n}\\
 		z^{n+\epsilon-m}(a^{(1)}_{0,0,n+\epsilon-m}t_1+a^{(2)}_{0,0,n+\epsilon-m}t_2+a^{(3)}_{0,0,n+\epsilon-m}t_3)&z^{\epsilon}(b^{(1)}_{0,0,\epsilon}t_1+b^{(2)}_{0,0,\epsilon}t_2+b^{(3)}_{0,0,\epsilon}t_3)
 	\end{matrix}\right]\in P.$$
 }
 
  We need to analyze two sub-cases:
 
 \underline{Subcase 2.1.} $m-n>0.$
 
 Here, $\theta_1:=t_1z^{\epsilon}(b^{(1)}_{0,0,\epsilon}t_1+b^{(2)}_{0,0,\epsilon}t_2+b^{(3)}_{0,0,\epsilon}t_3)\in P.$ 
 We claim that $\theta_1\neq 0.$ Otherwise, $b^{(1)}_{0,0,\epsilon}=b^{(2)}_{0,0,\epsilon}=b^{(3)}_{0,0,\epsilon}=0.$ In particular,  $I_1(\phi)\subset (x^m,(x,y)y^n).$ Thus, $$I_1(\phi)_{(x,y)}\subset (x^m,(x,y)y^n)_{(x,y)}$$
 which is properly contained in $(x^m,y^n)_{(x,y)},$ contradicting assumption (c).  Hence, $x,y,\theta_1$ is a regular sequence contained in $P.$ But, this is an absurd because   $R[t_1,t_2,t_3]/\mathcal{J}$ is a Cohen-Macaulay ring of codimension $2.$

 \underline{Subcase 2.2.} $m-n=0.$
 	We see that $\theta\neq 0.$ Indeed, otherwise, $a^{(1)}_{0,0,\epsilon}=b^{(2)}_{0,0,\epsilon}$ and $a^{(2)}_{0,0,\epsilon}=a^{(3)}_{0,0,\epsilon}=b^{(1)}_{0,0,\epsilon}=b^{(3)}_{0,0,\epsilon}=0.$ These equalities imply that neither $\Delta$ nor $p_3$  has $z$-degree lower than ${\epsilon},$ contradicting  assumption (d).
 	 Since $\theta$ does not involve either $x$ or $y$, then $\{x,y,\theta\}$ is a regular sequence contained in $P.$ 
 	 Once more, an absurd.
 	
 	This proves the claim.
 	
 	Passing to the ring $R[t_1,t_2,t_3,x^{-1}]$, due to the shape of (\ref{H-B2}), one clearly has $$\mathcal{J}R[t_1,t_2,t_3,x^{-1}]=(f,g)R[t_1,t_2,t_3,x^{-1}].$$
 	
 	On the other hand, writing $R[x^{-1}]\simeq R[T]/(xT-1)$, by the universal property of the symmetric algebra, one has 
 	$${\mathcal{S}}_R(I){\otimes}_R  R[T]/(xT-1)\simeq {\mathcal{S}}_{R[x^{-1}]} (IR[x^{-1}]).
 	$$
 	Thus, 
 	$$\frac{R[t_1,t_2,t_3,x^{-1}]}{\mathcal{J}R[t_1,t_2,t_3,x^{-1}]}=\frac{R[t_1,t_2,t_3,x^{-1}]}{(f,g)R[t_1,t_2,t_3,x^{-1}]}\simeq{\mathcal{S}}_{R[x^{-1}]} (IR[x^{-1}]).$$
 	But, by the shape of $\phi$,  $IR[x^{-1}]$ is itself generated by a regular sequence, hence is an ideal of linear type. In particular, ${\mathcal{S}}_{R[x^{-1}]} (IR[x^{-1}])$ is a domain.  But,  since $x$ is regular over $R[t_1,t_2,t_3]/\mathcal{J}$ and $R[t_1,t_2,t_3,x^{-1}]/\mathcal{J}R[t_1,t_2,t_3,x^{-1}]$ is a domain we conclude that $R[t_1,t_2,t_3]/\mathcal{J}$ is a domain as well.
 	Therefore, it must be the Rees algebra of $I$.
 \end{proof}

 \subsection{A role of the Sylvester forms}

 Since the rational map defined by the $2$-minors of $\phi$ is  no longer automatically confluent with the identity map of $\pp^1$, in particular there is no obvious analogue of a  downgraded sequence as mentioned in Remark~\ref{Rees_algebra_downgraded} and, consequently, no predictability for the minimal number of generators of the homogeneous defining  ideal of the Rees algebra of the ideal $I$.
 
 {\sc Sylvester forms.}
 We resort to the more encompassing notion of a Sylvester form. 
 
 For the basic facts on this notion we refer to either \cite[Section 4.1]{Syl2} or \cite[Subsection 2.1]{SimToh2015}. 
 In particular, the references \cite{BuSiTo2016}, \cite{HasSim2012}, \cite{Syl1}, \cite{LinShen}, and \cite{SimToh2015}  contain a vast amount of results as  to whether (or when), given a reasonably structured ideal, the homogeneous defining ideal of its Rees algebra is generated via iterated Sylvester forms.
 We note that the content matrix in the proof of Proposition~\ref{Zaq} is one of these Sylvester forms.
 
 We emphasize that, though parts of the above references deal quite a bit with the case where $R/I$ is Artinian (i.e., $I$ is $\fm$-primary), here we stick to the case where $\dim R/I=1$, which seems to bring up a significant difference.
 
 We let $t_1,t_2,t_3$ as before denote presentation variables over $R$ for the Rees algebra of $I=I_2(\phi)$.
 Let $\mathcal{I}\subset R[t_1,t_2,t_3]$ as before denote the corresponding presentation ideal.
 The next two examples show that Proposition~\ref{Zaq} fails if any one of assumptions (c) or (d) is missing.
 
 \begin{Example}\label{deg4}\rm
 	Let
 	\[ \varphi = \left[ \begin{matrix}
 		x^2 & yz \\
 		y^2 & xz \\
 		0 & y^2
 	\end{matrix}\right].
 	\]
 \end{Example}
 Here $m=2$, while the entries on the second column have degree $n+\epsilon=1+1=2$.
 Clearly, $I_1(\varphi)=(x^2,xz,yz,y^2)$ has an embedded component, while its codimension two component is $(x,y)$. 
 Thus, assumption (c) of Proposition~\ref{Zaq} fails here, while (d) still holds, with $\epsilon=1$.
 
 Note that the  radical of $I_1(\phi)$ is also $(x,y)$.
 The other radical that will come in is the one of $I$, namely, $(xz,y)$. 

 {\sc Claim.} The presentation ideal $\mathcal{I}$ of the Rees algebra $\mathcal{R}(I)$ is minimally generated by $\{f,g, h_1,h_2\}$, where $\{f,g\}$ are the generators induced by the columns of $\phi$, while $h_1$ (respectively, $h_2$) is a Sylvester determinant of bidegree $(2,2)$ (respectively, $(1,3)$).
 
 The gist of the claim is that, though one can trivially verify by computer assistant  that $\mathcal{I}$ is minimally generated by two more generators other than $\{f,g\}$, it won't tell you that these can be accomplished by means of two Sylvester determinants.
 
 In any case, it will at least tell us the bidegrees of the two additional bigraded minimal generators.
 As it is, the two bidegrees thus found are $(2,2)$ and $(1,3)$.
 What we will do over here is reach a compromise by doing half of the theoretic details, namely, we prove that $h_1$ and $h_2$ have these bidegrees, respectively, hence must account for the remaining bigraded minimal generators.
  
 The Sylvester form of $\{f,g\}$ with respect to $\{x,y\}$ (generating the radical of $I_1(\phi)$) is 
 $$h_1:=\det 
 \left[\begin{matrix} 
 	t_1x & t_2y\\
 	t_2z & t_3y+t_1z
 	 \end{matrix}\right]
 =t_1t_3xy+t_1^2xz-t_2^2yz.
 $$
As usual, $\mathcal{J}:=(f,g,h_1)$ is contained in the  ideal $\mathcal{I}$. However, $\mathcal{J}$ is not a prime ideal since it is contained in $(x,y)R[t_1,t_2,t_3]$, hence it is properly contained in $\mathcal{I}$.

 For the record,  $\mathcal{J}$ is the ideal of maximal minors of the matrix
	\[ \left[ \begin{matrix}
	-t_3y-t_1z & -t_2z\\
	t_2y  &   t_1x\\
	x   &   -y
\end{matrix} \right],
\] 
hence is Cohen--Macaulay.

We are brave and try yet another Sylvester determinant. This is a matter of fine choice to reach one of bidegree $(1,3)$ as the computer assessment guaranteed.
We take the Sylvester form of $\{g, h_1\}$ with respect to $\{xz,y\}$ (that generates the radical of $I$ this time around).
One gets:

$$h_2=\det 
\left[\begin{matrix} 
	t_2 &  t_3y+t_1z\\
	t_1^2 &  t_1t_3x-t_2^2z
\end{matrix}\right]
=t_1t_2t_3x-t_1^2t_3y-t_1^3z-t_2^3z.
$$
For the record, one also has $\mathcal{I}=\mathcal{J}:(x,y)$.

 Finally, $\mathcal{I}$ is not Cohen--Macaulay. For, if it were so then its minimal generators above, up to elementary operations, would be the ideal of maximal minors of a $4\times 3$ matrix.
 But this is impossible with the present total degrees of these same generators.
 
 The generators of $I_2(\phi)$ do not define a birational map as there is a unique minimal generator of $\mathcal{I}$ of bidegree $(1,b)$ ($b\geq 1$).
 
 \medskip
 
 \begin{Example}\label{degree6}
 	Let
 \[ \varphi = \left[ \begin{matrix}
 	x^2 & x^3z+y^4 \\
 	y^2 & x^4+y^3z \\
 	0 & x^4+y^4
 \end{matrix}\right].
 \]
 \end{Example}
 Here $m=2$, while the second column has degrees $n+\epsilon=2+2=4$.
 As one easily sees, $I_1(\phi)=(x^2,y^2)$, hence item (c) of the proposition holds, but (d) is violated since $\epsilon=2$.
 
 The claim over here is pretty much the same as in the former example -- in the sense that it boils down to two additional Sylvester determinants, of bidegrees $(2,2)$ and $(1,3)$.
 But this is all there is in common, as the respective radicals of $I_1(\phi)$ and of $I$ are not the right containing modules to draw upon.
 
 The Sylvester form of $\{f,g\}$ with respect to $\{x^2,y^2\}$ is 
 $$h_1=\det 
 \left[\begin{matrix} 
 	t_1 &  t_2x^2+t_3x^2\\
 	t_2 &  t_1y^2+t_3y^2-t_2xz+t_2yz
 \end{matrix}\right]
 =-t_2^2x^2-t_2t_3x^2+t_1^2y^2+t_1t_3y^2-t_1t_2xz+t_1t_2yz.
 $$
 To proceed, the ideal $(x-y,y^2)\supset (x^2,y^2)$ is the $(x,y)$-primary component of the ideal $(f,g,h_1)$, so in particular it contains $(f, h_1)$.
 We then consider the Sylvester form of $\{f,h_1\}$ with respect to $\{x-y,y^2\}$:
 
 \begin{eqnarray*}
 	h_2&=&\det 
 \left[\begin{matrix} 
 t_1x+t_1y & t_2^2x+t_2t_3x+t_2^2y+t_2t_3y+t_1t_2z\\
 	t_1+t_2   & -t_1^2+t_2^2-t_1t_3+t_2t_3
 \end{matrix}\right]\\
 &=& -t_1^3x-t_2^3x-t_1^2t_3x-t_2^2t_3x-t_1^3y-t_2^3y-t_1^2t_3y-t_2^2t_3y-t_1^2t_2z-t_1t_2^2z.
 \end{eqnarray*}

 The questions to follow are apparently open  even in the case where the entries of  $\phi$ are monomials and have standard `neighboring' degrees  $\left \lfloor{d/2}\right \rfloor$ and $\left \lceil{d/2}\right \rceil$, respectively, where $d$ is the common degree of the three minors (such as in Proposition~\ref{Zaq}).
 Set $I=I_2(\phi)$ and assume as above that $I_1(\phi)$ has codimension two and, in addition, its codimension two primary component is a complete intersection (of two forms).
 Finally, we keep the assumption that $d_2\geq d_1\geq 2$.

 \begin{Question}\rm
 	When is the Rees algebra of $I$  Cohen--Macaulay? 
 \end{Question}
 
 \begin{Question}\rm
 	Is $\mathcal{I}$ minimally generated by the two syzygetic relations $\{f,g\}$ of $I$ and iterated Sylvester forms starting with $\{f,g\}$?
 	Do the minimal generators that can be taken as Sylvester forms have different bidegrees, that is, for each occurring bidegree there is only one minimal generator of this bidegree?
 \end{Question}
 
 \begin{Question}\rm
 	Is the rational map defined by the standard generators of $I_2(\phi)$  ever birational beyond the special case of Proposition~\ref{deJonq}?
 \end{Question}

 
 \section{Codimension two perfect ideals generated in two degrees}\label{Section3}

 In this part we take over the ``next'' case, namely, assume that the homogeneous height two perfect ideal is generated in two different degrees.
 We state the preliminaries in terms of the associated Hilbert--Burch like matrix.
 
 \subsection{Basic players}\label{Basic}

 Let $R$ be a standard graded polynomial ring $R=k[x_1,\ldots,x_d]$ over a field and let $n\geq 2$ be an integer with a partition $n=a+(n-a)$, where $1\leq a\leq n-1$. Given integers $\epsilon_1, \epsilon_2\geq 1$, introduce the $n\times (n-1)$ block matrix
 \begin{equation}\label{basic_matrix}
 	\phi=\left[\begin{matrix}
 		\,	\Phi_1\,\\
 		\hline \,
 		\,	\Phi_2 \,
 	\end{matrix}\right],
 \end{equation}
 where $\Phi_1$ is an $a\times (n-1)$ submatrix  whose row entries have degree $\epsilon_1$ and $\Phi_2$ is an $(n-a)\times (n-1)$  matrix whose row entries  have degree $\epsilon_2$.
 We assume that the ideal $I_{n-1}(\phi)\subset R$ of maximal minors has codimension $\geq 2$ (hence, $2$).
 
 Set $I:=I_{n-1}(\phi)$. Since $n-a\leq n-1$,  the minimal graded free resolution of $I$ has the form 
 \begin{equation}\label{two-degree-resolution}
 	0\to R(-D)^{n-1}\stackrel{\phi}\lar  R(-(D-\epsilon_2))^{n-a}\oplus R(-(D-\epsilon_1))^a\to I\to 0,
 \end{equation}
 where $D: =a\epsilon_1 + (n-a)\epsilon_2$.
 
 To fix ideas, one might assume that $\epsilon_1\geq \epsilon_2$, so that the initial degree of $I$ turns out to be $D-\epsilon_1$. 
 We note that an important special case is that in which $\epsilon_1=\epsilon_2=1$ -- the case where $\phi$ is {\em linear}.
 
 \subsection{Minors fixing a submatrix}
 
 Let $J\subset I$  stand for the ideal generated by the maximal minors $f_1,\ldots,f_a$ of $\phi$  fixing the submatrix $\Phi_2$. Note that $J$ is generated in the initial degree $D-\epsilon_1$ of $I$.
 
 A free resolution of $R/J$ was obtained formerly in \cite[Theorem D]{AnSi1986} disregarding grading considerations  (similar result for arbitrary matrices and sizes remains apparently open -- see \cite[Conjecture 6.4.15]{SimisBook}).
 
 Here we make a digression, by emphasizing the role of the $R$-module $I/J$.
 For this, assume that $J$ has codimension two as well, which is the typical assumption of  \cite[Theorem D]{AnSi1986}.
 One can see in this case that $I/J$ is isomorphic to $\coker \Phi_2$ (more exactly, $I/J \simeq \coker \Phi_2 (-(D-\epsilon_2))$ is a homogeneous isomorphism), hence its free $R$-resolution relates to the one of $R/J$, as in \cite[Theorem D]{AnSi1986}. For the reader's convenience we reformulate the argument of  \cite[Theorem D, (ii) $\Rightarrow$ (i)]{AnSi1986} in the present environment.
 
 Set $R(-D)^{n-1}\stackrel {\Phi_1}\lar R(-(D-\epsilon_1))^{a}$ and
 $R(-D)^{n-1}\stackrel{\Phi_2}\lar R^{n-a}(-(D-\epsilon_2))$ for the two maps induced by the graded map defined by $\phi$.

 \begin{Proposition}\label{resij_implies_resrj} Let 
{\small 	$$0\to \bigoplus_{j} R(-j)^{\beta_{r,j}}\stackrel{\psi_{r-1}}\lar\cdots\stackrel{\psi_2}\lar\bigoplus_j R(-j)^{\beta_{2,j}}\stackrel{\psi_1}\lar R(-D)^{n-1}\stackrel{\Phi_2}\lar R^{n-a}(-(D-\epsilon_2))\to I/J\to 0$$
}
 	be a minimal graded free resolution of the $R$-module $I/J.$ Then, the minimal graded free resolution of $R/J$ is
 	{\small\begin{equation}\nonumber
 			0\to \bigoplus_{j} R(-j)^{\beta_{r,j}}\stackrel{\psi_{r-1}}\lar\cdots\stackrel{\psi_2}\lar\bigoplus_j R(-j)^{\beta_{2,j}}\stackrel{\Phi_1\psi_1}\lar R(-(D-\epsilon_1))^{a}\stackrel{\mathfrak{f}}\to R\to  R/J\to 0,
 	\end{equation}}
 	where $\mathfrak{f}=\left[\begin{matrix}
 		f_1& \cdots&f_a\end{matrix}\right],$
and $\Phi_1\psi_1$ is the composite of $\psi_1$ and $\Phi_1$.

	In particular, $\depth R/J=\depth I/J$.
 \end{Proposition}
 \begin{proof} Due to the above shape of the free resolution of $I/J$, it suffices to show that $$\ker(\mathfrak{f})=\Image(\Phi_1\,\psi_1)\quad \mbox{and}\quad\ker (\Phi_1\,\psi_1)=\Image\psi_2.$$ 
 	
 	Since an element of $\ker(\mathfrak{f})$ is a syzygy of the minors $f_1,\ldots,f_a$,  the first of the above equalities follows from the equivalences
 	\begin{eqnarray*}
 		\eta\in\ker(\mathfrak{f})&\Leftrightarrow&
 		\left[\begin{matrix}
 			\eta\\\boldsymbol0_{(n-a)\times 1}\end{matrix}\right]\in\Image \phi\subset R^n\\
 		&\Leftrightarrow& \phi\,\uu=\left[\begin{matrix}
 			\eta\\\boldsymbol0_{(n-a)\times 1}\end{matrix}\right], \;\mbox{for some } \uu\in R^{n-1}\\
 		&\Leftrightarrow& \left[\begin{matrix}
 			\Phi_1\,\uu\\ \Phi_2\,\uu\end{matrix}\right]=\left[\begin{matrix}
 			\eta\\\boldsymbol0_{(n-a)\times 1}\end{matrix}\right],\;\mbox{for some } \uu\in R^{n-1}\\
 		&\Leftrightarrow&\Phi_1\,\uu=\eta\,\,\mbox{and}\,\, \Phi_2\,\uu=0,\;\;\mbox{for some } \uu\in R^{n-1}\\\
 		&\Leftrightarrow&\Phi_1\,\uu=\eta\,\,\mbox{and}\,\,\uu\in\Image{\psi_1}, \;\mbox{for some } \uu\in R^{n-1}\\
 		&\Leftrightarrow& \eta\in\Image(\Phi_1\,\psi_1).
 	\end{eqnarray*}
 	For the second equality, we have $\Phi_1\,\psi_1\,\psi_2=\boldsymbol0.$
 	Hence, $\Image(\psi_2)\subset\ker(\Phi_1\,\psi_1).$ Now, let $\zeta\in\ker(\Phi_1\,\psi_1),$  so $\Phi_1\,\psi_1\,\zeta=0.$ Since $\ker\Phi_2=\Image (\psi_1),$ then also $\Phi_2\,\psi_1=0.$ Thus, $$\phi\,\psi_1\zeta=\left[\begin{matrix}
 		\Phi_1 \Phi_2\end{matrix}\right]\psi_1\zeta=0.$$
 	Since $\ker\phi=\boldsymbol0$ then $\psi_1\zeta=\boldsymbol0$, hence in particular, $\zeta\in \Image(\psi_2).$ Thus, $\ker (\Phi_1\,\psi_1)=\Image(\psi_2)$ as desired.
 \end{proof}
 
 \subsection{A role of the Buchsbaum--Rim complex}
 
 We recapitulate in the present environment how the  minimal graded free resolution of the $R$-module $I/J$ is  determined.
 
 Quite generally, let $R$ be a Noetherian ring and $\psi: F\to G$ be a map of free modules over a ring $R.$  Write $s:=\rank F$ and $r:=\rank G$ and suppose that $s\geq r.$ The {\it Buchsbaum-Rim complex} of $\psi$ has the following well-known shape
 	\begin{eqnarray}\label{BuchRim}
 		0\to ({\mathcal{S}}_{s-r-1} (G))^{\ast}\otimes \wedge^{s}F 
 		& \stackrel{\partial}\lar  & ({\mathcal{S}}_{s-r-2}(G))^{\ast}\otimes \wedge^{s-1}F \\ \nonumber &\stackrel{\partial}\lar &\cdots\stackrel{\partial}\lar \wedge^{r+1} F   
 		\stackrel{\theta}\lar   F\stackrel{\Psi}\lar G.
 	\end{eqnarray}
 By choosing bases of the  free modules, we can partly make explicit the nature of its differentials (see \cite{Eis}[Appendix 2] for further details):
 
 \begin{Remark}\label{diferentials}\rm
 	Let $\{e_1,\ldots,e_s\}$ denote a free basis of $F.$ Given indexes $1\leq i_1<\ldots<i_{r+1}\leq s$, with $e_{i_1}\wedge\cdots\wedge e_{i_{r+1}}\in \bigwedge^{r+1}F$, then the map $\theta$ is given by
 	$$
 	\theta(e_{i_1}\wedge \cdots\wedge e_{i_{r+1}})= \sum_{j=1}^{r+1}(-1)^{r+1-j}\det(\phi_j) e_{i_j},$$
 	where $\phi_j$ is the $r\times r$ submatrix of $\psi$ with columns corresponding to the basis elements indexed by $i_1<\ldots i_{j-1} < i_{j+1} \cdots <i_{r+1}$. 
 	
 	Moreover, the  entries of  any differential $\partial$ are $\mathbb{Z}$-linear forms in the entries of $\psi.$
 \end{Remark}

 \begin{Proposition}\label{BR_main}{\rm (\cite[Section A2.6]{Eis})}
 	The Buchsbaum--Rim complex of $\psi$ is a free resolution of $\coker\psi$ if and only if $\grade I_r(\psi)\geq s-r+1.$
 \end{Proposition}
 At our end, the  interest lies in the graded case of the map $\psi$. Precisely, we consider the following landscape, for whose details we claim no priority:
 \begin{Proposition}\label{BRimparticularcase}
 	Let $R=k[x_1,\ldots,x_d]$ be a standard graded polynomial ring over a field $k$ and let  $R(-\boldsymbol{\delta})^{s}\stackrel{\psi}{\longrightarrow} R^{r}$ denote a graded  map satisfying $\Ht I_{r}(\psi)=s-r+1$. Write $M:=\coker\psi$.
 	Then$:$
 	\begin{enumerate}
 		\item[{\rm (i)}] The minimal graded free resolution of $M$ has the form
 			\begin{eqnarray*}
 				0\to R(-s\boldsymbol{\delta})^{\beta_{s-r-1}}\stackrel{\partial}\lar R(-(s-1)\boldsymbol{\delta})^{\beta_{s-r-2}} \stackrel{\partial}\lar\cdots \to R(-(r+2)\boldsymbol{\delta})^{\beta_1} \\
 				\stackrel{\partial}\lar R(-(r+1)\boldsymbol{\delta})^{\beta_0}\stackrel{\theta}\lar R(-\boldsymbol{\delta})^{s}\stackrel{\psi}\lar R^{r} \to M\to 0.
 			\end{eqnarray*}
 		where $\beta_i={r-1+i\choose i}{s\choose i+r+1}$ for each $0\leq i\leq s-r-1.$
 		\item[{\rm (ii)}] If, moreover, $r=s-2,$ then the minimal graded free resolution of $M$ has the form 
 		\begin{equation*}
 			0\to R(-4\boldsymbol{\delta})^{s-2} \stackrel{\psi^t}\lar R(-3\boldsymbol{\delta})^{s} \stackrel{\theta}\lar R(-\boldsymbol{\delta})^{s}\stackrel{\psi}\lar R^{s-2} \to M\to 0,
 		\end{equation*}
 		where $\theta$ is the $s\times s$ skew-symmetric matrix
 		\begin{equation*}
 			\left[\begin{matrix}
 				0&(-1)^{s+r-2}\delta_{1,2}&\ldots&(-1)^{r}\delta_{1,s}\\
 				(-1)^{s+r-1}\delta_{1,2}&0&\ldots&(-1)^{r-1}\delta_{2,s}\\
 				\vdots&\vdots&\ddots&\vdots\\
 				(-1)^{s-1}\delta_{1,s}&(-1)^{s-2}\delta_{2,s}&\ldots&0
 			\end{matrix}\right].
 		\end{equation*}
 		Here the non-null entries of the  $j$th column of $\theta$ are the ordered signed $(s-2)$-minors of the submatrix of $\psi$  obtained by omitting its $j$th column. In particular, one has a matrix equality 
 		\begin{equation*}
 			\TT\theta=	\left[\begin{matrix}
 				\Delta_1&\cdots&\Delta_s\end{matrix}\right]
 		\end{equation*}
 		where $\TT=\left[\begin{matrix}
 			T_1&\cdots&T_{s}\end{matrix}\right]$ is a vector of indeterminates over $R$ and $\Delta_1,\ldots,\Delta_s$ are the ordered signed $(s-1)$-minors of the $s\times (s-1)$ matrix $\left[\begin{matrix}
 			\TT^t&\psi^t\end{matrix}\right].$
 	\end{enumerate}
 \end{Proposition}
 
 \begin{proof}
 	(i) From (\ref{BuchRim}), one has $({\mathcal{S}}_{i}(G))^{\ast}\otimes \wedge^{i+r+1}F\simeq R^{{r-1\choose i} }\otimes R^{{s\choose i+r+1}}\simeq R^{\beta_i}.$ 
 	Then Remark~\ref{diferentials} and Proposition~\ref{BR_main}  imply the stated  minimal graded free resolution of $M$.
 	
 	(ii) As a special case of (i), the minimal graded free resolution of $M$ has the following shape 
 	
 	\begin{equation*}
 		0\to R(-4\boldsymbol{\delta})^{s-2} \stackrel{\psi_1}\lar R(-3\boldsymbol{\delta})^{s} \stackrel{\theta}\lar R(-\boldsymbol{\delta})^{s}\stackrel{\psi}\lar R^{s-2} \to M\to 0.
 	\end{equation*}
 	In order to make $\psi_1$ explicit, consider  respective free bases  $$\{(-1)^{s-1}e_2\wedge\cdots\wedge e_s,\ldots, (-1)^{s-i}e_1\wedge\cdots\wedge\widehat{e}_i\wedge\cdots\wedge e_s,\ldots,(-1)^{(s-s)}e_1\wedge\cdots\wedge e_{s-1}\}$$
 	and $\{e_1,\ldots,e_s\}$ of  $ R(-(r+1)\boldsymbol{\delta})^s\simeq \wedge^{s-1} F(-(r+1)\boldsymbol{\delta})$   and $F(-\boldsymbol{\delta})=R(-\boldsymbol{\delta})^{s}$.
 	Then, Remark~\ref{diferentials} implies the claimed format of $\theta$. Since  $\psi \theta=\boldsymbol0$ and $\theta$ is skew-symmetric then $\theta\psi^t=\boldsymbol0.$ Thus, we have the complex
 	\begin{equation}\label{resMr=s-2}
 		0\to R(-4\boldsymbol{\delta})^{s-2} \stackrel{\psi^t}\lar R(-3\boldsymbol{\delta})^{s} \stackrel{\theta}\lar R(-\boldsymbol{\delta})^{s}\stackrel{\psi}\lar R^{s-2} \to M\to 0.
 	\end{equation}
 	Since $\Ht I_{r}(\psi)=3$,  the known Buchsbaum-Eisenbud acyclicity criterion implies that (\ref{resMr=s-2}) is acyclic.
 \end{proof}
 
 \subsection{A main theorem}
 
 We now rephrase some of the previous free resolutions in the case where $\Ht I_{n-a}(\Phi_2)=a$.
 
 \begin{Corollary}\label{ResofJ}
 	If $\Ht I_{n-a}(\Phi_2)=a$ then the respective minimal graded free resolution of $I/J$ and $R/J$ have the following form
 	{\small
 		\begin{eqnarray}
 			0\to R(-((n-2)\epsilon_2+D))^{\beta_{a-2}}\to R(-((n-3)\epsilon_2+D))^{\beta_{a-3}}\to\cdots\to R(-((n-a+1)\epsilon_2+D))^{\beta_1} \nonumber\\\to R(-((n-a)\epsilon_2+D))^{\beta_0}\stackrel{\theta}\lar R(-D)^{n-1}\stackrel{\Phi_2}\lar R^{n-a}(-(D-\epsilon_2))\to I/J\to 0\nonumber
 	\end{eqnarray}}
 	and
 	{\small
 		\begin{eqnarray}\nonumber
 			0\to R(-((n-2)\epsilon_2+D))^{\beta_{a-2}}\to R(-((n-3)\epsilon_2+D))^{\beta_{a-3}}\to\cdots\to R(-((n-a+1)\epsilon_2+D))^{\beta_1}\\ \nonumber
 			\to R(-((n-a)\epsilon_2+D))^{\beta_0} \stackrel{\Phi_1\theta}\lar R(-(D-\epsilon_1))^{a}\to R\to R/J\to 0\nonumber
 	\end{eqnarray}}
 	where   $\beta_i={n-a-1+i\choose i}{n-1\choose i+n-a+1}$ for  $0\leq i\leq a-2$.
 \end{Corollary}
 \begin{proof} Since  $\Ht I_{n-a}(\Phi_2) =a=(n-1)-(n-a)+1$ then, by Proposition \ref{BR_main} the minimal graded free resolution of $\coker\Phi_2$ is
 	
 	{\small
 		\begin{eqnarray}
 			0\to R(-(n-1)\epsilon_2)^{\beta_{a-2}}\to R(-(n-2)\epsilon_2)^{\beta_{a-3}}\to\cdots\to R(-(n-a+2)\epsilon_2)^{\beta_1} \nonumber\\\to R(-(n-a+1)\epsilon_2)^{\beta_0}\to R(-\epsilon_2)^{n-1}\stackrel{\Phi_2}\lar R^{n-a}\to \coker \Phi_2\to 0\nonumber
 	\end{eqnarray}}
 			But, in this case, $I/J\simeq \coker\Phi_2(-(D-\epsilon_2)) .$ Hence, the minimal graded free resolution of $I/J$ it is as stated. The minimal graded free resolution of $R/J$ follows from the minimal graded free resolution of $I/J$ and Proposition~\ref{resij_implies_resrj}.
 		\end{proof}
 		
 		
 		The hypothesis $\Ht I_{n-a}(\Phi_2)=a$ in the above corollary imposes $1\leq a\leq d,$ where $d=\dim R$. Obviously, $J$ is a principal ideal if $a=1$ and a complete intersection if $a=2.$ Thus, the interest lies on the range $3\leq a\leq d.$ 
 		In this range, if $\Ht I_{n-a}(\Phi_2)=a$ then the ideal $J$ is not perfect since $\dim R/J=d-2$, while $\depth R/J=d- {\rm hd}_R(R/J)\leq d-3$.
 		
 		We tacitly assume that $\Ht I_{n-a}(\Phi_2)=a$ for the rest of the section.
 		Throughout will consider the extreme values $a=3$ and $a=d$ in some more detail.
 		
 		\smallskip

 		{\Large $\bullet \;a=3$}
 		
 		\smallskip

 		The free resolution of the ideal $J$ follows from Corollary~\ref{ResofJ}. We next look at its symmetric algebra.
 		Letting $S=R[T_1,T_2,T_3]=R[\TT]$ be a standard polynomial ring over  $R=k[x_1,\ldots,x_d]$,  we know that the symmetric algebra ${\mathcal{S}}_R(J)$ of $J$ over $R$ is presented over $S$ by the ideal  $I_1(\TT \Phi_1\,\theta)$ generated by the entries of the matrix product $[T_1\; T_2\: T_3]\, \Phi_1\,\theta$. 
 		\begin{Proposition} Assume that $a=3.$ If $\Ht I_{n-3}(\Phi_2)=3$ then the symmetric algebra ${\mathcal{S}}_R(J)$ is a Cohen-Macaulay $S$-module with  minimal  graded free resolution 
 			{\small
 				\begin{eqnarray*}
 					0 &\to & S(-((n-3)\epsilon_2+2\epsilon_1+2)) \oplus S(-((n-2)\epsilon_2+\epsilon_1+1))^{n-3}\\ 
 					&\to& S(-((n-3)\epsilon_2+\epsilon_1+1))^{n-1} \to S\to {\mathcal{S}}_R(J)\to0.
 			\end{eqnarray*}}
 		\end{Proposition}
 		\begin{proof} (i) By the Huneke-Rossi formula \cite[Theorem 2.6 (ii)]{HR} we have $$\dim{\mathcal{S}}_R(J) = \sup_{P\in\spec{R}}\{\mu(J_P)+\dim R/P\}.$$
 			If $\Ht P\geq 2$ then $$\mu(J_P)+\dim R/P\leq 3+(d-2)=d+1.$$
 			If $\Ht P=1$ then
 			$$\mu(J_P)+\dim R/P= 1+(d-1)=d.$$
 			If $P=(0)$ then
 			$$\mu(J_P)+\dim R/P=d+1.$$
 			Thus, $\dim{\mathcal{S}}_R(J)=d+1.$

 			It follows that $\Ht I_1(\TT \Phi_1\,\theta)=2.$ By Proposition~\ref{BRimparticularcase}(ii),  $I_1(\TT \Phi_1\,\theta)$ is the ideal generated by the $(n-2)$-minors of the $(n-1)\times(n-2)$ matrix $\left[\begin{array}{c|c}
 				(\TT \Phi_1)^t&\Phi_2^{t}\end{array}\right].$ Hence, by the Hilbert-Burch theorem, $I_1(\TT \Phi_1\theta)$ is a perfect ideal of height $2$ with free resolution
 			$$0\to S^{n-2}
 			\stackrel{[(\TT \Phi_1)^t |  \Phi_2^{t}]}{\longrightarrow} S^{n-1}\to I_1(\TT \Phi_1\theta)\to 0.$$
 			Inspection of the  degrees in the involved matrices gives the remaining degrees of  the minimal graded free resolution of ${\mathcal{S}}_R(J)$ as stated.
 		\end{proof}

 		{\Large $\bullet \;a=d$ and $\epsilon_2=1$}
 		
 		\smallskip
 		
 		We assume throughout that $\Phi_2\neq 0$, i.e., that $\phi$ has effectively nonzero linear entries.
 		In other words, $a\leq n-1$, and hence in the present case, $n-d\geq              1$.
 		
 		We need the following lemma, which runs independently and is mostly well-known:
 		
 		\begin{Lemma}\label{BR_basic_equivalences}
 			Let $R=k[x_1,\ldots,x_d]$ denote as previously a standard graded polynomial ring of a field $k$ and let  $R(-1)^{n-1} \stackrel{\psi}{\to} R^{n-d}$ be a graded linear map. The following assertions are equivalent:
 			\begin{enumerate}
 				\item[{\rm(i)}] $I_{n-d}(\psi)=(x_1,\ldots,x_d)^{n-d}.$
 				\item[{\rm(ii)}]  $\Ht I_{n-d}(\psi)=d.$
 				\item[{\rm(iii)}]  The Buchsbaum--Rim complex of $\psi$ is acyclic.
 				\item[{\rm(iv)}] The Eagon--Northcott complex of $I_{n-d}(\psi)$ is acyclic.
 			\end{enumerate}
 		\end{Lemma}
 		\begin{proof}
 			(i) $\Rightarrow$ (ii) This is obvious.
 			
 			(ii) $\Leftrightarrow$ (iii) This is Proposition~\ref{BR_main}.
 			
 			(ii) $\Rightarrow$ (iv) This is well-known.
 			
 			(iv) $\Rightarrow$ (i) Since the Eagon--Northcott complex is a minimal free graded resolution of $I_{n-d}(\psi)$, the minimal number of generators of the latter is ${{n-1}\choose {n-d}}$, which is the minimal number of generators of $(x_1,\ldots,x_d)^{n-d}.$
 		\end{proof}

 		\smallskip
 		
 		Recall the notion of {\em chaos invariant} from \cite[Section 2]{linpres2018}.

 		\begin{Proposition} With the notation and assumption of {\rm \ref{Basic}}, assume that $a=d.$\label{saturation_equivs}
 			The following are equivalent:
 			\begin{enumerate}
 				\item[{\rm(a)}] $J^{\rm sat}=I.$
 				\item[{\rm(b)}] $\Ht I_{n-d}(\Phi_2)=d.$
 				\item[{\rm(c)}] The Buchsbaum--Rim complex of $\Phi_2$ is a graded minimal free resolution of $I/J$. 
 			\end{enumerate}
 			Moreover, if $d=3$ and $\epsilon_1 =1$ then these conditions are equivalents to$:$
 			\begin{enumerate}
 				\item[{\rm(d)}] The chaos invariant  of $I$ is $\geq n-3.$
 			\end{enumerate}
 		\end{Proposition}
 		\begin{proof} First, recall that the matrix $\Phi_2$ is the syzygy matrix of the $R$-module $I/J.$ Thus, since the annihilator $0:_RI/J$ and the zeroth Fitting ideal $I_{n-d}(\Phi_2)$ of $I/J$ have the same radical then $$\Ht 0:_RI/J=\Ht I_{n-d}(\Phi_2).$$
 			With this equality and the fact that $J^{\rm sat}= I$ if and only if $I/J$ is $(x_1,\ldots,x_d)$-primary  as an $R$-module, the equivalence of (a) and (b) follows through.
 			The equivalence of (b) and (c) follows from Lemma~\ref{BR_basic_equivalences}.
 			
 			The equivalence of (d) and (b) under the stated hypotheses follows immediately from the definitions.
 		\end{proof}
 		
 		Next is the main theorem of this subsection.
 	
 	\begin{Theorem}\label{main-thm}
 		With the previous notation, setting $R=k[x_1,\ldots,x_d]$, assume that $a=d$ and that the equivalent conditions of {\rm Proposition~\ref{saturation_equivs}} hold. 
 		Then$:$
 		\begin{enumerate}
 			\item[{\rm (a)}] $J$ satisfies $G_{\infty}$ if and only if $I$ satisfies $G_d.$
 			\item[{\rm(b)}] If $\epsilon_1=\epsilon_2=1$ then $J$ is a reduction of $I.$
 			\item[{\rm(c)}] $J$ is of linear type if and only if $I$ satisfies $G_d.$
 		\end{enumerate}
 	\end{Theorem}
 	\begin{proof} (a) By hypothesis we have $J^{\rm sat}=I.$ Thus, $J_P=I_P$ for every prime ideal $P\neq (x_1,\ldots,x_d).$ With this and the assumption that $\mu(J)=d$ the stated equivalence follows through.
 		
 		(b) Letting $f_1,\ldots,f_n$ be the ordered signed $(n-1)$-minors of the basic matrix $\phi$ of (\ref{basic_matrix}), one has $J=(f_1,\ldots,f_d).$
 		Since the rows of $\phi$ are syzygies of  $f_1,\ldots,f_n$, we can write
 		\begin{equation*}
 			\left[\begin{matrix}
 				f_{d+1}&\cdots&f_n
 			\end{matrix}\right]\Phi_2=-\left[\begin{matrix}f_1&\cdots& f_d\end{matrix}\right]
 			\Phi_1.
 		\end{equation*}
 		Thus, for any  $(n-d)\times(n-d)$ submatrix $B$ of $\Phi_2$, one has
 		$$I_1(\left[\begin{matrix}
 			f_{d+1}&\cdots&f_n
 		\end{matrix}\right]B)\subset ( x_1,\ldots, x_d) J.$$
 		Bringing in the adjugate matrix, one can write $$I_1(\left[\begin{matrix}
 			f_{d+1}&\cdots&f_n
 		\end{matrix}\right] B\,{\rm adj}(B))\subset ( x_1,\ldots, x_d)^{n-d}J,$$
 		and hence, $(\det B)(f_{d+1},\ldots,f_n)\subset ( x_1,\ldots, x_d)^{n-d}J.$ Since $B$ is arbitrary, it obtains 
 		\begin{equation*}\label{I_{n-d}IsubsetI_{n-d}J}
 			I_{n-d}(\Phi_2)( f_{d+1},\ldots,f_n)=( x_1,\ldots, x_d)^{n-d}( f_{d+1},\ldots,f_n)\subset ( x_1,\ldots, x_d)^{n-d} J,
 		\end{equation*}
 		where the equality follows from Lemma~\ref{BR_basic_equivalences} since $\Phi_2$ is assumed to be linear.
 		
 		This inclusion is an expression of integral dependence.
 		In fact, let $M_1,\ldots,M_N$ denote the set of monomials of $R$  of degree $n-d.$ Then
 		$$M_jf_i=L_{j,1} M_1+\cdots +L_{j,N} M_N, \quad 1\leq j\leq N, \, d+1\leq i\leq n,$$
 		for certain $k$-linear forms $L_{j,l}\in k[f_1,\ldots,f_d] \subset R$ over $f_1,\ldots,f_d$.
 		These relations can be read as
 		$$\left(f_i\mathbb{I}_N- (L_{j,l})_{1\leq j,l\leq N}\right)\left[\begin{matrix}M_1&\cdots&M_N\end{matrix}\right]^t=\boldsymbol0,$$
 		where $\mathbb{I}_N$ denotes the $N\times N$ identity matrix,
 		or yet, 
 		$$\det\left(f_i\mathbb{I}_N-(L_{j,l})_{1\leq j,l\leq N}\right)( x_1,\ldots,x_d)^{n-d}=0.$$
 		Hence, 
 		$$\det\left(f_i\mathbb{I}_N-(L_{j,l})_{1\leq j,l\leq N}\right)=0$$
 		which is an equation of integral dependence of $f_i$ over $J=(f_1,\ldots,f_d)$.  In particular, $J$ is a reduction of $I.$

 		(c)   Suppose that $I$ satisfies $G_d.$
 		
 		\medskip
 		
 		{\sc Claim.} The symmetric algebra ${\mathcal{S}}_R(J)$ is a Cohen-Macaulay ring.
 		
 		\medskip
 		
 		Specialize the $n\times (n-1)$ generic matrix $Y=(y_{i,j})$ over $k$  to $\phi$ via mapping $y_{i,j}$ to the corresponding entry of $\phi$.
 		The kernel of this map is generated by a regular sequence of length $n(n-1)-d$ in $T:=k[y_{i,j}]$.
 		Consider the corresponding ideal $J'\subset T$ generated by the $(n-1)$-minors of $Y$ fixing the last $(n-d)$ rows.
 		By \cite[Theorem C (ii)]{AnSi1986}, the symmetric algebra ${\mathcal{S}}_T(J')$ is a Cohen-Macaulay ring of dimension $n(n-1)+1.$
 		Therefore, it is enough to show that $\dim {\mathcal{S}}_R(J)=\dim {\mathcal{S}}_T(J')-[n(n-1)-d]=d+1.$ But, by  (a), $J$ satisfies $G_{\infty}.$ Thus, by the Huneke-Rossi formula \cite[Theorem 2.6 (ii)]{HR}, $\dim {\mathcal{S}}_R(J)=d+1.$ 
 		This warps up the claim.
 		
 		
 	To conclude, since ${\mathcal{S}}_R(J)$ is Cohen-Macaulay and $J$ satisfies $G_{\infty}$  we conclude by \cite[Proposition 8.5]{Trento} that ${\mathcal{S}}_R(J)\simeq \mathcal{R}_R(J).$ 
 		
 		The converse is clear.
 	\end{proof}
 	
 	As a consequence of the preceding theorem, we prove the conjectured statement in \cite[Conjecture 3.1]{BuSiTo2022}  in the case of a generic arrangement.
 	
 	 \begin{Theorem}\label{Arrang_conj} {\rm (char$(k)\nmid n$)}
 		Let $F\in R=k[x_1,\ldots,x_d]$ stand for the defining form of a generic hyperplane arrangement of size $n\geq d+1$ and let $J_F$ denote its gradient ideal. Then, $J_F$ is an ideal of linear type. 
 	\end{Theorem}
 	\begin{proof} Let $I\subset R$ be the ideal generated by the  $(n-1)$-fold products of the  equations of the hyperplanes of $F$. We collect here some well known facts about the ideal $I:$
 		\begin{enumerate}
 			\item[(1)] The ideal $I$ is a perfect ideal of codimension 2 generated by the $(n-1)$-minors of an $n\times(n-1)$ matrix $\phi$ of linear forms of $R.$
 			\item[(2)] The ideal $I$ satisfies $G_{d}$ (see the proof of \cite[Proposition 4.1(a)]{GST}).
 		\end{enumerate}
 		According to \cite[Proof of Theorem 2.4]{RTY2024}, after elementary row/column operations, we can decompose $\phi$ in the following way:
 		\begin{equation}
 			\phi=\left[\begin{matrix}
 				\,	\Phi_1\,\\
 				\hline \,
 				\,	\Phi_2 \,
 			\end{matrix}\right],
 		\end{equation}
 		such that $\Phi_2$ is a $(n-d)\times n$ matrix with $\Ht I_{n-d}(\Phi_2)=d$ and $J_F$ is the ideal generated by the $(n-1)$-minors of $\phi$ fixing the rows of $\Phi_2.$ Thus, by the Theorem~\ref{main-thm}(c) we conclude that $J_F$ is of linear type.
 	\end{proof}

\section{A player in the theory of plane reduced points}

The defining ideal of a finite set of reduced points in projective $2$-space is a distinguished case of a perfect codimension two ideal.
In this part we revisit some aspects of this everlasting theory, with emphasis on a certain model of generic behavior.
This model will be confronted with previous models discussed by Geramita and co-authors.

\subsection{Preliminaries}

Let $X=\{p_1,\ldots,p_n\}$ be a set of $n$ distinct points in $\pp^2=\pp^2_k={\rm Proj}(k[x,y,z]),$ where $k$ is assumed to be algebraically closed -- or else, the points have rational coordinates in the algebraic closure of $k$. Set $R=k[x,y,z]$. Let $I(X)\subset R$ denote the homogeneous reduced (radical) ideal of $X$, so
$$I(X)=P_1\cap\cdots\cap P_n,$$
where $P_j\subset R$  is the homogeneous prime ideal of the point $p_j.$
In particular, $I(X)$ is generically a complete intersection, which means, in this dimension, that $I(X)$ satisfies  condition $G_3$.

Clearly,  $I:=I(X)$ is a graded perfect ideal of height two, hence has a minimal graded free resolution of the following form
\begin{equation}\label{res_I(X)}
0\rar F_1=\bigoplus_{j\geq 0}R(-j)^{\beta_{1,j}}\to F_0=\bigoplus_{j\geq 0} R(-j)^{\beta_{0,j}}\to I\to 0,
\end{equation}
for suitable shifts. 
A bit unprecisely, the subsumed matrix defining the map $F_1 \rightarrow F_0$ is a Hilbert--Burch matrix associated to the ideal $I$.
The following well-known facts are stated for convenience.

\begin{Lemma}\label{data_res} With the above notation,  one has$:$
	\begin{enumerate}
		\item[{\rm (a)}] $\beta_{0,j}=\dim_k (I_j)-\dim_k(R_1 I_{j-1}).$
		\item[{\rm (b)}] $\beta_{1,j}-\beta_{0,j}=-\Delta^3 H_I(j)$ where $\Delta$ denotes the difference operator and $H_I$ is the Hilbert function of $I.$
	\end{enumerate}
\end{Lemma} 

It follows from (a) and (b) that the minimal graded free resolution \eqref{res_I(X)} is completely determined by $F_0$ and the Hilbert function $H_I$ of $I.$ In general, 
\begin{equation}\label{bound_hilbert}
H_I(t)\geq \max\left\{0,{t+2\choose 2}-n\right\}\quad\quad \mbox{for every } t\geq 0. 
\end{equation}

The following terminology was introduced by Geramita et al (\cite{GO1}, \cite{GO2}, \cite{GM}):

\begin{Definition}\rm
	Let $I\subset R$ be an ideal of reduced  points $X=\{p_1,\ldots,p_n\}$. 
	
	(a) $X$ is  said to be in {\it generic $n$-position} if the above condition \eqref{bound_hilbert} is an equality for every $t\geq 0.$
	
	(b) $X$ is  said to be  in {\it uniform {\rm (}generic{\rm )} position} if, for any number $1\leq m\leq n$, every subset of $X$ with $m$ elements is in generic $m$-position.
\end{Definition}
Clearly, if $X$ is a set of $n$ points in uniform generic position then it is in generic $n$-position, but not conversely as examples show (Example~\ref{redone} and Example~\ref{not_uniform_pos}).

The above are sub-conditions of the vague condition that the set of points $\{p_1,\ldots,p_n\}$ is {\em general} in the sense of lying on a dense Zariski open set of a corresponding parameter space, viewing  a set of $n$ points of $\pp^2$ as a point in multi-projective space $\pp^2\times\cdots\times\pp^2=(\pp^2)^n.$ 
As such, it was shown in \cite[Theorem 4]{GO1} that the set  of points of $(\pp^2)^n$  in generic $n$-position form a dense open subset of $(\pp^2)^n$ -- this result is actually valid in arbitrary dimension.

Set $\mathfrak{G}_n$ for the set  of (tuples of) points of $(\pp^2)^n$  in generic $n$-position.
Hoping it will cause no confusion, elements of $\mathfrak{G}_n$ will be referred to as points.

The following result of \cite{GO2} will play a central role  in the sequel:

\begin{Proposition}\label{GO2} Let $I\subset R=k[x,y,z]$ be the reduced ideal of a point $X\in \mathfrak{G}_n$.
Then:
\begin{enumerate}
	\item[{\rm (a)}] $s:={\rm indeg}(I)$ is the least integer such that $n< {s+2\choose2}.$ In particular, ${s+1\choose2}\leq n$, so one can write $n={s+1\choose2}+h$ for some $0\leq h\leq s;$ conversely, if $s\geq 1$ is an integer such that $n={s+1\choose 2}+h$, with $0\leq h\leq s$, then $s$ is the initial degree of $I$.
	\item[{\rm (b)}] {\rm (\cite[Corollary 3]{GO2})} $I=(I_s,I_{s+1})$ where $s={\rm indeg}(I).$  
\end{enumerate} 
\end{Proposition}

Next is how the regularity works in the above landscape.

\begin{Proposition}\label{regI}
	 Let $I\subset R=k[x,y,z]$ be the reduced  ideal of a point  $X\in\mathfrak{G}_n$. If $s\geq 2$ is the initial degree of $I$, then:
	$$\reg I=
	\left\{\begin{array}{cc}
	s,&\mbox{if } n={s+1\choose 2}\\
	s+1,&\mbox{otherwise},
	\end{array}\right.$$
	where $\reg I$ denotes the Castelnuovo-Mumford regularity of $I.$
\end{Proposition}
\begin{proof} 
	By Proposition~\ref{GO2} (a), one can write $n={s+1\choose 2}+h$, with $0\leq h\leq s.$

Since $R/I$ is a finitely generated graded Cohen-Macaulay $R$-module (of dimension $1$), then  the smallest index $r$ such that $H_{R/I}(t)=n$ for all $t\geq r$ is  $$r=1-\depth R/I+\reg R/I=\reg R/I$$ (see for example, \cite[Corollary 4.8]{EisSyz}). On the other hand, since $X$ is a set of points in $n$-generic position in $\pp^2$ then
\begin{eqnarray}
H_{R/I}(t)&=&{t+2\choose 2}-\max\left\{0,{t+2\choose 2}-n\right\}\nonumber\\
&=&\left\{\begin{array}{cc}{t+2\choose 2}&\mbox{if } t\leq s-1\\ n&\mbox{if } t\geq s.\end{array}\right.
\end{eqnarray}
Putting together, one has  $\reg R/I=s-1$ if $h=0$ and $\reg R/I=s$ if $h\neq 0.$ Hence, $\reg I=(s-1)+1=s$ if $h=0$ and $\reg I=s+1$ if $h\neq 0.$
\end{proof}

\subsection{Points in tight generic position}

Let $\mathfrak{G}_n\subset (\pp^2)^n$ as before denote the set of the (tuples of) points in generic $n$-position.
The following result, conjectured in \cite{LRoberts}, was obtained as a consequence of  \cite[Theorem 2.6]{GM}, and later extended to arbitrary dimensions in \cite{TrVa1989}.

\begin{Theorem}\label{tight-is-open} Let  $\mathfrak{TG}_n\subset\mathfrak{G}_n\subset (\pp^2)^n$ denote the subset of the  points whose corresponding reduced ideals $I\subset R$ satisfy $\dim_k(R_1 I_s)=\min\{3\dim_k I_s,\dim_k I_{s+1}\}$, with $s={\rm indeg}(I)$.
Then, $\mathfrak{TG}_n$ contains a non-empty open subset.
\end{Theorem}

\begin{Definition}\rm A point (tuple) belonging to the subset  $\mathfrak{TG}_n$ in the above theorem will be said to be in {\it tight generic $n$-position}.	
\end{Definition}
We stress that this terminology refers to this very particular sort of point distribution, not to points chosen at random (although the latter would certainly fit).

The basic dimensions under the above stronger hypothesis are given as follows:

\begin{Proposition}\label{minimal_values}
	Let $I\subset R=k[x,y,z]$ be the reduced  ideal of a point  $X\in \mathfrak{TG}_n$.
	 Write $n={s+1\choose2}+h$ for some $0\leq h\leq s$, according to {\rm Proposition~\ref{GO2} (a)},  where $s={\rm indeg}(I)$.
	Then$:$
	\begin{enumerate}
		\item[{\rm (i)}] $\dim_k I_s=s-h+1$ and $\dim_k I_{s+1}=2s+3-h$.
		\item[{\rm (ii)}] $\dim_k(R_1 I_s)=\min\{3(s-h+1),2s+3-h\}=\left\{\begin{array}{cc}2s+3-h&\;\mbox{if }0\leq h\leq s/2\\3(s-h+1)&\,\mbox{if }s/2< h\leq s.\end{array}\right.$
	\end{enumerate}
\end{Proposition}
\begin{proof} (i) Since $X$ is in generic $n$-position, then
$$\dim_k I_t= \max\left\{0,{t+2\choose 2}-n\right\}$$
for every $t\geq 0.$ In particular,

\begin{eqnarray}
\dim_k I_s&=& \max\left\{0,{s+2\choose 2}-n\right\}\nonumber\\
&=&\max\left\{0,{s+2\choose 2}-{s+1\choose 2}-h\right\}\nonumber\\
&=&\max\{0,s-h+1\}=s-h+1\nonumber
\end{eqnarray}
and
\begin{eqnarray}
\dim_k I_{s+1}&=& \max\left\{0,{s+3\choose 2}-n\right\}\nonumber\\
&=&\max\left\{0,{s+3\choose 2}-{s+1\choose 2}-h\right\}\nonumber\\
&=&\max\{0,2s+3-h\}=2s+3-h.\nonumber
\end{eqnarray}

(ii) Since $X$ is in {\em tight} generic $n$-position, this is an immediate consequence of (i).
\end{proof}

\begin{Theorem}\label{res-ideal-gen-points}
	Let $I\subset R=k[x,y,z]$ be the reduced  ideal of a point  $X\in \mathfrak{TG}_n$. Write $n={s+1\choose2}+h$ for some $0\leq h\leq s,$ where $s={\rm indeg}(I)\geq 2$. Then the minimal graded free resolution of $I$ is:
	\begin{equation}\label{first_half}
	0\to R(-s-1)^{s-2h}\oplus R(-s-2)^{h} \stackrel{\phi_h}\lar R(-s)^{s-h+1}\to I\to 0  
	\end{equation}\label{second_half}
	if  $0\leq h\leq s/2$ and
	\begin{equation}
	0\to R(-s-2)^{h} \stackrel{\phi_h}\lar  R(-s)^{s-h+1}\oplus R(-s-1)^{2h-s} \to I\to 0  
	\end{equation}
	if $s/2< h\leq s.$ 
\end{Theorem}
\begin{proof} By Proposition~\ref{regI}, the regularity of $I$ is $s$ or $s+1$. Therefore, the minimal graded free resolution of $I$ has the following shape:
$$0\to R(-s-1)^{\beta_{1,s+1}}\oplus R(-s-2)^{\beta_{1,s+2}}\to R(-s)^{\beta_{0,s}}\oplus R(-s-1)^{\beta_{0,s+1}}\to I\to 0,$$
for certain integers $\beta_{i,j}$, which we proceed to determine.

By  Lemma~\ref{data_res}~(a) and by Proposition~\ref{minimal_values}, one has

\begin{equation*}
\beta_{0,s}=\dim_k (I_s)-\dim_k(R_1 I_{s-1})=\dim_k (I_s)=s-h+1
\end{equation*}
and
\begin{eqnarray*}
\beta_{0,s+1}&=&\dim_k (I_{s+1})-\dim_k(R_1 I_{s})\nonumber\\
&=&2s+3-h-\dim_k(R_1 I_{s})\nonumber\\
&=&\left\{\begin{array}{cc}0&\mbox{if }0\leq h\leq s/2\\
2h-s&\mbox{if }s/2< h\leq s.\end{array}\right. \quad (\mbox{by Proposition~\ref{minimal_values}})
\end{eqnarray*}

By Lemma~\ref{data_res}~(b), one has
\begin{eqnarray*}
\beta_{1,s+1}&=&\beta_{0,s+1}-H_I(s+1)+3H_I(s)-3H_I(s-1)+H_I(s-2)\nonumber\\
&=&\beta_{0,s+1}-H_I(s+1)+3H_I(s)\nonumber\\
&=&\left\{\begin{array}{cc}s-2h&\mbox{if }0\leq h\leq s/2\\
0&\mbox{if }s/2< h\leq s\end{array}\right.
\end{eqnarray*}

\begin{eqnarray*}
\beta_{1,s+2}&=&\beta_{0,s+2}-H_I(s+2)+3H_I(s+1)-3H_I(s)+H_I(s-1)\nonumber\\
&=&-H_I(s+2)+3H_I(s+1)-3H_I(s)\nonumber\\
&=&h.
\end{eqnarray*}
Carrying over these values, one retrieves the claimed resolutions. 
\end{proof}

\smallskip

Next we call upon the standing inequality for a set of $n$ points in generic $n$-position, as established in Proposition~\ref{GO2}, namely:  $n={s+1\choose 2}+h$, with $0\leq h\leq s,$ where $s={\rm indeg}(I)\geq 3$.
The next two subsections will emphasize the value of $h$ as it sweeps two sectors, under the assumption of tight generic position.

\subsection{Tight generic position: range $0\leq h\leq s/2$}

\begin{Proposition}\label{Ideal_Ger_Birational}
Let $I\subset R=k[x,y,z]$ be the reduced ideal of a point  $X\in \mathfrak{TG}_n$. Write $n={s+1\choose 2}+h$, with $0\leq h\leq s/2,$ where $s={\rm indeg}(I)\geq 3$.
	Then$:$
	\begin{enumerate}
		\item[{\rm (a)}] The ideal $I$ has maximal analytic spread $\ell(I)$.
		\item[{\rm(b)}] If moreover $h\leq s/2-1$, the rational map $\mathfrak{F}:\pp^2\dasharrow \pp^{s-h}$ defined by the linear  system $I_s$ is birational onto its image.
	\end{enumerate} 
\end{Proposition}
\begin{proof}
(a)  Given that $s\geq 3$, the number of generators of $I$ is $s-h+1\geq 3$. Since $I$ is a radical ideal, it cannot be equimultiple according to the result of Dade and Cowsik--Nori (independently) (see \cite{Cow-Nor}).

(b)
The assumption $h\leq s/2-1$ means that $s-2h\geq 2$, hence the resolution (\ref{first_half}) in Theorem~\ref{res-ideal-gen-points} implies that $I$ has at least two independent linear syzygies. 
Let $\phi_1$ denote the $2\times (s-h+1)$ matrix whose columns are these two linear syzygies and let $\mathcal{R}_R(I)\simeq R[t_1,\ldots,t_{s-h+1}]/\mathcal{J}$ be a bihomogeneous polynomial presentation of the Rees algebra of $I.$ Let $\{f,g\}\subset \mathcal{J}$ denote the two forms of bidegree $(1,1)$ induced by the two linear syzygies. The Jacobian matrix $\Theta$ of $f$ and $g$ with respect to the variables $x,y,z$ is a $2\times 3$ submatrix of the Jacobian dual matrix of $I.$  
In order to deduce that the map is birational onto the image it is enough, according to  \cite[Theorem 2.20]{AHA}, to prove that the rank of $\Theta$ module the defining ideal of the special fiber $\mathcal{F}(I)$ is $2$. But, since $I$ satisfies the $G_3$ condition the  initial degree of the special fiber of $I$ is at least $3$ (see \cite[Corollary 1.3]{TightRS2022}). Thus, to conclude the proof we need to show that the rank of $\Theta$ is $2$ over $S:=k[t_1,\ldots,t_{s-h+1}].$ For this, let  $E:=\Image \phi_1.$ Then, $E$ is an $R$-module of projective dimension $1$ with free resolution 
$$0\to R^2\stackrel{\phi_1}\lar R^{s-h+1}\to E\to 0.$$

\noindent{\sc Claim:} The symmetric algebra $\mathcal{S}_R(E)$ of $E$ is a domain.

Since $E$ is an $R$-module of projective dimension $1$ then, in order to proof the claim it is enough, by  \cite[Theorem 1.1]{Hu}, to show that $\Ht I_2(\phi_1)\geq 2$ and $\Ht I_1(\phi_1)\geq 3$.

Note that $\Ht I_2(\phi_1)\geq 2$ because $I\subset I_2(\phi_1).$ As for $\Ht I_1(\phi_1),$ suppose otherwise. Then $I_1(\phi_1)=p$ for some minimal prime $p$ of $I.$ Thus,  $I\subset I_2(\phi_1)\subset p^2.$ Hence, $I_p=p\subset p_p^2.$ But, this an absurd. Therefore $\Ht I_1(\phi_1)=3$, thus concluding the proof of the claim.

Set $M:=\coker(\Theta^t).$
By the equality $$[\begin{matrix}t_1 &\cdots&t_{s-h+1}\end{matrix}]\phi_1=[\begin{matrix}x&y&z\end{matrix}]\Theta^t,$$
$M$ is naturally a module over the polynomial ring $S=k[t_1,\ldots, t_{s-h+1}]$ and we have the usual isomorphism of $R\otimes_k S=k[x,y,z,t_1,\ldots, t_{s-h+1}]$-modules between the respective symmetric algebras 
$\mathcal{S}_R(E)$ and $\mathcal{S}_S(M).$ 
In addition, since  $\mathcal{S}_R(E)$ is a domain, then $\rank M+ \mu(E)=\rank E+\mu(M)$ (\cite[Proposition 1.1]{JacDual}.
Since $\rank E=s-h+1$, $\mu(E)=s-h+1$ and $\mu(M)=3$, then $\rank M=1$.
That is, $\rank \Theta=3-1=2$, as was to be shown.
\end{proof}

\begin{Remark}\rm 
	Among the cases  left out in (b), one finds $(s,h)=(3,1)$ and $(s,h)=(4,2)$. Now in both cases, by Theorem~\ref{res-ideal-gen-points},  the ideal $I$ is $3$-generated.
	Since it is generically a complete intersection then it is of linear type.
	Looking at the respective free resolutions, one sees that the linear rank of $I$ is at most $1$. 
	It follows from \cite[Proposition 3.4]{AHA} that the rational map  $\mathfrak{F}$ is not birational (i.e., not Cremona).
\end{Remark}
	
	

Again, in the presence of the strict inequality $h \leq s/2-1$, one can say a bit more.

\begin{Theorem}\label{strict_inequality}
	Let $I \subset R = k[x,y,z]$ be a reduced ideal of a point $X \in \mathfrak{TG}_n$ with $\mu(I) \geq 4$. If $n = {s+1 \choose 2} +h$, with $0 \leq h \leq s/2-1,$ the following are equivalent$:$
	\begin{enumerate}
		\item[\rm(a)]  The Rees algebra $\mathcal{R}_R(I)$ is Cohen-Macaulay.
		\item[\rm(b)] The special fiber $\mathcal{F}(I)$ is Cohen-Macaulay and the minimal graded free resolution of $\mathcal{F}(I)$ over $S=k[t_1,\ldots,t_{\mu(I)}]$ has the form
		$$0 \rightarrow S(-\mu(I)+1)^{\beta_{\mu(I)-3}} \rightarrow \cdots \rightarrow S(-3)^{\beta_1} \rightarrow S \rightarrow \mathcal{F}(I) \rightarrow 0$$
		\item[\rm(c)] $I$ is is linearly presented $($that is, $h=0$$).$
	\end{enumerate}
\end{Theorem} 
\begin{proof}(a)$\Rightarrow$(b) Since $I$ is an ideal of reduced points, it satisfies condition $G_3$. Thus, it follows by \cite[Corollary 1.3]{TightRS2022}  that ${\rm indeg}(Q)\geq 3$. In addition, as $\ell(I)=3$ (by Proposition \ref{Ideal_Ger_Birational} (a)) and $\mu(I)\geq 4$ then by  \cite[Theorem 3.3]{TightRS2022} the special fiber $\mathcal{F}(I)$ is Cohen-Macaulay. 
	
	Since $\mathcal{R}(I)$ is Cohen-Macaulay, then  \cite[Theorem 5.6]{AHT1995} yields $r(I) \leq \ell(I)-1= 2$, where $r(I)$ denotes the minimal reduction number of $I$. Since $\mathcal{F}(I)$ is Cohen-Macaulay,  it follows by \cite[Proposition 1.2]{GST}  that ${\rm reg}(\mathcal{F}(I)) = r(I)$. Thus, ${\rm reg}(\mathcal{F}(I)) \leq 2.$ On the other hand, since ${\rm indeg}(Q)\geq 3$  we have $\reg(\mathcal{F}(I)) \geq 2.$ Therefore,  $\reg(\mathcal{F}(I)) =2$  and the defining ideal of $\mathcal{F}(I)$ is generated in degree $3.$ In particular  the minimal graded free resolution of $\mathcal{F}(I)$ has the form
	$$0 \rightarrow S(-\mu(I)+1)^{\beta_{\mu(I)-3}} \rightarrow \cdots \rightarrow S(-3)^{\beta_1} \rightarrow S \rightarrow \mathcal{F}(I) \rightarrow 0.$$
	
	
	(b)$\Rightarrow$(c) By the assumed minimal graded free resolution of $\mathcal{F}(I)$ and  \cite[Theorem 1.2]{HM}, one has $e(\mathcal{F}(I)) = {\mu(I)-1 \choose 2}.$ By Theorem~\ref{res-ideal-gen-points}, $\mu(I)=s-h+1.$ Thus,
	\begin{equation}\label{graudafibra}
		e(\mathcal{F}(I))= \frac{s^2-2sh-s+h^2+h}{2}.
	\end{equation}
	
	Since $\ell(I)=3 \neq 2 = \Ht(I)$ and $I$ is generically a complete intersection, then by \cite[Theorem 6.6 (b)]{Ram2} we have
	\begin{equation}\label{BirReducedEquigWithFiberCM}
		e(\mathcal{F}(I)) \cdot \deg(\mathfrak{F})  = s^2 - e(R/I)=\frac{s^2-s-2h}{2},
	\end{equation}
	where $\deg(\mathfrak{F})$ is the degree of the rational map $\mathfrak{F}:\pp^2\dasharrow\pp^{\mu(I)-1}$ defined by the linear system $I_s.$ But, by Proposition~\ref{Ideal_Ger_Birational}, $\deg(\mathfrak{F})=1.$ Hence, by \eqref{graudafibra} and \eqref{BirReducedEquigWithFiberCM} we get 
	\begin{equation*}
		\frac{s^2-2sh-s+h^2+h}{2} = \frac{s^2-s-2h}{2},
	\end{equation*}
thus implying that $h(-2s+h+3)=0$. Supposing that $h \neq 0$, we would have  $-2s+h+3=0$, that is, $2s-3=h < s/2$, hence $s < 2$, which contradicts $\mu(I)=s-h+1\geq 3$. Thus, $h=0$ and, therefore $I$ is a linearly presented ideal.

	(c)$\Rightarrow$(a) Since $I$ is a  perfect ideal of codimension 2  that is generically complete intersection then $I$ satisfies the $G_3$ condition. Hence, by \cite[Theorem 1.3]{MorUl1996} the Rees algebra $\mathcal{R}(I)$ is Cohen-Macaulay.
\end{proof}

\subsection{Tight generic position: range $s/2 < h\leq s$}\label{second_range}

The following general result wraps up various situations considered in \cite[Section 4]{GO2}.

\begin{Proposition}\label{irreducible_everywhere}
	Let $I\subset R=k[x,y,z]$ be the reduced ideal of a point  $X\in \mathfrak{TG}_n$.  Let $n={s+1\choose 2}+h$, with $s/2 < h\leq s,$ where $s={\rm indeg}(I)\geq 2$. If $X$ is in addition in uniform $n$-position then any element of the linear system $I_s$  is irreducible. In particular, the ideal $(I_s)$ has height $2$.
\end{Proposition}
\begin{proof} Let $f\in I_s$, $f\neq 0$. Suppose that $f=f_1f_2$ with $0<s_i:=\deg f_i\leq s-1$ for $i=1,2.$ Since $I$ is radical, we can assume $\gcd (f_1,f_2)=1.$ 
	
	Clearly, $X\subset V(f)=V(f_1)\cup V(f_2).$ In particular, $X=(X\cap V(f_1))\cup(X\cap V(f_2)).$ Let $m_i$ denote the number of elements of $X\cap V(f_i)$, for $i=1,2.$ 
	
	Since  $f_i\in I(X\cap V(f_i))$ then ${\rm indeg \,}I(X\cap V(f_i))\geq s_i$ and since $X$ is in uniform position, it follows from  Proposition~\ref{GO2} that $m_i\leq {{s_i+2}\choose 2}-1.$ Thus,
	\begin{eqnarray}
		n&=& m_1+m_2-\#(X\cap V(f_1)\cap V(f_2))\nonumber\\
		&\leq& m_1+m_2 \nonumber\\
		& \leq&{s_1+2\choose 2}+ {s_2+2\choose 2}-2\nonumber\\
		&=&\frac{s_1^2+s_2^2+3(s_1+s_2)+4}{2}-2\nonumber\\
		&=&\frac{s^2+3s-2s_1s_2+4}{2}-2 \quad\quad(\mbox{because } s_1+s_2=s)\nonumber\\
		&=&{s+1\choose 2}+s-s_1s_2.\label{estimativas1s2}
	\end{eqnarray}
	Since $s_1+s_2=s$ then $s_1\geq s/2$ or $s_2\geq s/2.$ Hence, $s_1s_2\geq s/2.$  In particular, $s-s_1s_2\leq s/2<h.$ Thus,  it follows from \eqref{estimativas1s2}  that 
	$$n<{s+1\choose 2}+h.$$
	But, this is an absurd because $n= {s+1\choose 2}+h$. 
	
	To see the supplementary statement, if $\Ht (I_s)=1$, since $R$ is an UFD, every two $s$-forms in $I$ have a proper common factor, which contradicts the main statement.
\end{proof}

\smallskip

Although the full hypotheses of the proposition look a bit over-sized, the following example shows that the additional assumption of uniform position is in general essential in order that the ideal $(I_s)$ should  have maximum height.

\begin{Example}\label{redone}\rm
	Let $I\subset k[x,y,z]$ be the ideal generated by the $2$-minors of the matrix
	$$\left[\begin{array}{cc}
		3yz+3z^2 &  -y^2+z^2\\
		y^2-4yz   &   x^2-y^2\\
		-x       &      0
	\end{array}\right].$$
	A computer calculation yields that $I$ is the reduced ideal of a point  $X\in \mathfrak{TG}_8$, with $s=3, h=s-1=2$, hence $n=8$. Note that this information is compatible with the data in Proposition~\ref{minimal_values}.
	Obviously, $(I_3)\subset (x)$, so both the main and the supplementary statements fail.
	Alas, $X$  is not in uniform position as there are subsets with more than $2$ points lying on a straight line.
\end{Example}

Let $\mathfrak{L}$ denote the $(2h-s)\times h$ submatrix of the syzygy matrix of $I$ whose entries are of degree $1$.
Then the ideal of minors $I_{2h-s}(\mathfrak{L})$ has height at most $h-(2h-s)+1=s-h+1$. But as we are in dimension $3$, attaining this upper bound implies $h=s-2$.
In this regard, one has:

\begin{Theorem}\label{main-sector2}
	Let $I\subset R= k[x,y,z]$ denote the reduced ideal of a point  $X\in \mathfrak{TG}_n$. Write $n={s+1\choose 2}+h$, with $h=s-2,$ where $s={\rm indeg}(I)\geq 5$. Let $J\subset I$ denote the ideal generated by the linear system $I_s.$  Then$:$
	\begin{enumerate}
		\item[{\rm(a)}] If $X$ is in addition in uniform $n$-position, the following are equivalent:
		\begin{itemize}
			\item[{\rm (i)}] $\dim I/J=0$.
			\item[{\rm (ii)}] The ideal of minors  $I_{s-4}(\mathfrak{L})$ has codimension $3$.
			\item[{\rm (iii)}] 	$J^{\rm sat}=I.$
		\end{itemize}
		\item[{\rm(b)}] 
		If any of the equivalent conditions in item {\rm (a)} is satisfied, the following hold$:$
		\begin{itemize}
			\item[{\rm (i)}] $J$ is an ideal of linear type.
			\item[{\rm (ii)}] The Rees algebra $\mathcal{R}_R(J)$ is Cohen--Macaulay.
			\item[{\rm (iii)}] The rational map $\mathfrak{F}: \pp^2 \dashrightarrow \pp^{\mu(J)-1}$ defined by $J$ is not birational onto its image.  
			\item[{\rm (iv)}] $deg(\mathfrak{F}) =  \frac{s^2-3s+4}{2}$.
		\end{itemize}
	\end{enumerate}
\end{Theorem}
\begin{proof}
	(a) By Theorem~\ref{res-ideal-gen-points} the minimal graded free resolution of $I$ is
	\begin{equation}
		0\to R(-s-2)^{s-2} \stackrel{\phi}\lar  R(-s)^{3}\oplus R(-s-1)^{s-4} \to I\to 0.  
	\end{equation}
	Let $f_1,f_2,f_3$ (respectively, $\{g_1,\ldots,g_{s-4}\}$) denote the minimal homogeneous generators of $I$ of degree $s$  (respectively, the minimal homogeneous generators of $I$ in degree $s+1$ not belonging to $R_1I_s$). In particular, $J=(f_1,f_2,f_3),$ an ideal of height $2$ by Proposition~\ref{irreducible_everywhere} (ii).

	Write the syzygy matrix $\phi$ as
	$$\phi=\left[\begin{array}{cc}
		\mathfrak{Q}\\ \textbf{}\mathfrak{L}
	\end{array}\right],$$
	where $\mathfrak{Q}$ is a $3\times (s-2)$ matrix whose entries are $2$-forms and $\mathfrak{L}$ is an $(s-4)\times (s-2)$ matrix whose entries are $1$-forms.
	
	(i) and (ii) are obviously equivalent since the annihilator of $I/J$ is such that $\Ht J:I=\Ht I_{s-4}(\mathfrak{L})$ because $I/J$ is presented by $\mathfrak{L}$, i. e., $I/J=\coker \mathfrak{L}$.
	
	By a similar token, the implication (iii) $\Rightarrow$ (i) holds.
	
	We now argue that (ii) $\Rightarrow$ (iii).

	Introducing a corresponding matrix notation
	$$\underline{f}=\left[\begin{array}{cccccccc}f_1&f_2&f_3\end{array}\right]\quad\mbox{and}\quad\underline{g}=\left[\begin{array}{cccccccc}g_1&\ldots&g_{s-4}\end{array}\right],$$
	one has
	$$\underline{g}\, \mathfrak{L}=-\underline{f}\, \mathfrak{Q}.$$
	Thus, $I_1(\underline{g}\, \mathfrak{L})=I_1(-\underline{f}\, \mathfrak{Q})\subset (f_1,f_2,f_3)=J.$ In particular, for an arbitrary $(s-4)\times (s-4)$ submatrix $\widetilde{\mathfrak{L}}$ of $\mathfrak{L}$ we have $I_1(\underline{g}\, \widetilde{\mathfrak{L}})\subset J.$ Hence,  by the adjugate relation
	$$(g_1\det \widetilde{\mathfrak{L}},\ldots,g_{s-4}\det \widetilde{\mathfrak{L}})=I_1(\underline{g}\, \widetilde{\mathfrak{L}}\,{\rm adj }\widetilde{\mathfrak{L}})\subset I_1(\underline{g}\, \widetilde{\mathfrak{L}})\subset J.$$
	Hence, the ideal product $I_{s-4}(\mathfrak{L})\, (g_1,\ldots,g_{s-4})$ is contained in $J$.
	It follows that
	\begin{equation}
		I_{s-4}(\mathfrak{L})\, I\subset J,
	\end{equation}
	thus implying that $I\subset J:(x,y,z)^{\infty}$.
	The reverse inclusion is obvious since $I$ is saturated.

	(b) Since $J$ is an almost complete intersection, the symmetric algebra $S_R(J)$ is Cohen-Macaulay (see, e.g., \cite[Corollary 10.4]{Trento}).
	On the other hand, by (a), the unmixed part of $J$ is the ideal $I$. The latter is generically a complete intersection since it is an ideal of reduced points. Therefore, $J$ is also generically a complete intersection and since it is $3$-generated, it satisfies property $(F_1)$ (or $G_{\infty}$).
	
	Combining the two results, it is well-known that $J$ is an ideal of linear type (see, e.g., \cite[Proposition 8.5]{Trento}).
	This shows (i), while (ii)  is an immediate consequence of (i) and its proof.

	(iii) Since $\Ht I_{s-4}(\mathfrak{L})=3$, then by Proposition~\ref{BRimparticularcase} and Corollary~\ref{ResofJ}  the minimal graded free resolution of  $R/J$ has the form
	
	\begin{equation}\label{resolution_J_three_gensOfDegrees}
		0\to R(-(2s-1))^{s-4}\lar R(-(2s-2))^{s-2}\lar R(-s)^3\to R \to R/J \to 0.
	\end{equation}
	Note that $J$ has no linear syzygies, therefore by \cite[Proposition 3.4]{AHA} it follows that the rational map  $\mathfrak{F}$ is not birational.
	
	(iv) By (\ref{resolution_J_three_gensOfDegrees}) we have that the Hilbert series of $R/J$ is given by
	\begin{equation*}
		H_{R/J}(t) = \frac{B_{R/J}(t)}{(1-t)^3}
	\end{equation*}
	where 
	\begin{equation}\label{NumeratorOfHilbertSeriesIn Dim3}
		B_{R/J}(t)=1-3t^s+(s-2)t^{2s-2}-(s-4)t^{2s-1}.
	\end{equation}
	By \cite[Corollary 7.4.12]{SimisBook} we have
	\begin{equation}\label{MultiplicityinSiBook}
		e(R/J) = \frac{1}{2}\frac{\partial^2 B_{R/J}(t)}{\partial t^2}(1).
	\end{equation}
	Thus, by (\ref{NumeratorOfHilbertSeriesIn Dim3}) and (\ref{MultiplicityinSiBook}) we get 
	\begin{equation}\label{MultiplicityOfJ}
		e(R/J) = \frac{s^2+3s-4}{2}.
	\end{equation}
	As $J$ is generically a complete intersection then by \cite[Theorem 6.6 (b)]{Ram2} we have 
	\begin{equation}\label{MultiplicityFor_d_Three}
		e(\mathcal{F}(J)) \cdot deg(\mathfrak{F}) = s^2-e(R/J).
	\end{equation}
	Since $J$ is linear type then $\mathcal{F}(J)$ is a polynomial ring, hence $e(\mathcal{F}(J))=1$. Then  (\ref{MultiplicityOfJ})and  (\ref{MultiplicityFor_d_Three}) give
	\begin{equation*}
		 \deg(\mathfrak{F}) = s^2-\frac{s^2+3s-4}{2}= \frac{s^2-3s+4}{2},
	\end{equation*}
	as was to be shown.
\end{proof}

We are not sure as to whether the equivalent conditions of (a) actually hold under the standing assumptions of the theorem.
In any case, the following example shows that the uniform position assumption is essential for the equivalent conditions of (a) to actually hold.

\begin{Example}\label{not_uniform_pos}\rm Let $I\subset R=k[x,y,z]$ be the ideal generated by the $3$-minors of the following matrix
	\begin{equation}\label{non-uniform}
		\left[\begin{array}{ccc}
			x^2&0&z^2\\
			y^2&x^2&0\\
			z^2&y^2&x^2\\
			0&z&y
		\end{array}\right].
	\end{equation}
	A calculation with \cite{M2} shows that $I$ is the reduced ideal of a set $X$ of  $18$ points in $\pp^2$.
	Computing its Hilbert function gives that $X$ is in generic $18$-position. Moreover, an additional calculation gives that $\dim_k R_1I_5=9$, hence $X$ is in tight $18$-position.
	However, it is not in uniform position since it has four distinct aligned points (on the line $y=0$).
	Direct inspection shows that $(y,z)\in {\rm Min}(R/J)$ but $(y,z)\notin {\rm Min}(R/I)$ -- actually, $J$ is the ideal of maximal minors of the matrix
	$$\left[\begin{array}{ccc}
		z^2 & -y^3+x^2z\\
		-y^2  & x^2y\\
		x^2 & z^3
	\end{array}\right].$$	
	
	In particular, $J^{\rm sat}=J\neq I.$
\end{Example}

\begin{Remark}\rm
	On the not so bright side, it would seem like a modification of the preceding example would have the ideal $I$ in tight uniform $n$-position and still not verifying  the equivalent conditions of item (a). Namely, replace the above $3\times 3$ upper submatrix by one with general entries of degree $2$.
\end{Remark}

  \bibliographystyle{amsalpha}

  \newpage
  
\noindent {\bf Addresses:}

\smallskip

\noindent {\sc Dayane Lira}\\
Departamento de Ciência e Tecnologia, DCT\\
Universidade Federal Rural do Semi-Árido\\
59780-000 Caraúbas, RN, Brazil\\
{\em e-mail}: dayannematematica@gmail.com\\


\noindent {\sc Geisa Oliveira}\\
Departamento de Matem\'atica, CCEN\\ 
Universidade Federal de Pernambuco\\ 
50740-560 Recife, PE, Brazil\\
{\em e-mail}: geisa.gama@ufpe.br\\


\noindent {\sc Zaqueu Ramos}\\
Departamento de Matem\'atica, CCET\\ 
Universidade Federal de Sergipe\\
49100-000 S\~ao Cristov\~ao, SE, Brazil\\
{\em e-mail}: zaqueu@mat.ufs.br\\


\noindent {\sc Aron Simis}\\
Departamento de Matem\'atica, CCEN\\ 
Universidade Federal de Pernambuco\\ 
50740-560 Recife, PE, Brazil\\
{\em e-mail}:  aron.simis@ufpe.br

\end{document}